\documentclass[11pt, a4paper]{amsart}
\usepackage{graphicx, xcolor}
\usepackage{tikz}
	\usetikzlibrary{arrows,shapes}

\tikzstyle{vertex}=[draw,circle,fill=black!5,minimum size=25pt,inner sep=0pt]
\tikzstyle{edge} = [draw,thick,<->]
\tikzstyle{weight} = [font=\large]

\newcommand{\R}{\mathbb{R}}

\newcommand{\N}{\mathbb{N}}
\newcommand\diag{\textrm{\normalfont diag}\,}
\newcommand\Real{\textrm{\normalfont Re}}

\newcommand\dive{\textrm{\normalfont div}\,}
\newcommand\linspan{\textrm{\normalfont span}\,}
\newcommand\dt{\textrm{\normalfont dt}}
\newcommand\dx{\textrm{\normalfont dx}}

\newtheorem{theorem}{Theorem}[section]
\newtheorem{proposition}[theorem]{Proposition}

\newtheorem{lemma}[theorem]{Lemma}
\newtheorem{corollary}[theorem]{Corollary}
\newtheorem{example}[theorem]{Example}

\begin{document}
\baselineskip=16pt

\thispagestyle{empty}

\begin{center}\sf
{\Large Velocity-jump processes with a finite number of speeds}\vskip.2cm
{\Large and their asymptotically parabolic nature}\vskip.25cm
{\tt \today}\vskip.2cm
Corrado MASCIA\footnote{Dipartimento di
  Matematica ``G. Castelnuovo'', Sapienza -- Universit\`a di Roma, P.le Aldo
  Moro, 2 - 00185 Roma (ITALY), \texttt{\tiny mascia@mat.uniroma1.it}
  {\sc and} {Istituto per le Applicazioni del Calcolo, Consiglio Nazionale delle Ricerche
  (associated in the framework of the program ``Intracellular Signalling'')}}
\end{center}
\vskip.5cm


\begin{quote}\footnotesize\baselineskip 12pt 
{\sf Abstract.} 
The paper examines a class of first order linear hyperbolic systems, proposed as a generalization
of the Goldstein--Kac model for velocity-jump processes and determined by a finite number of
speeds and corresponding transition rates.
It is shown that the large-time behavior is described by a corresponding scalar diffusive equation
of parabolic type, defined by a diffusion matrix for which an explicit formula is given.
Such representation takes advantage of a variant of the Kirchoff's matrix tree Theorem applied to the 
graph associated to the system and given by considering the velocities as verteces and the
transition rates as weights of the arcs.
\vskip.15cm

\vskip.15cm

\end{quote}

\section{Introduction}\label{sect:intro}

The {\bf Goldstein--Kac model} for correlated random walk (\cite{Gold51}, \cite{Kac74}) 
consists in a first order linear hyperbolic system for the couple $(f,g)=(f,g)(x,t)$ given by
\begin{equation*}
	\left\{\begin{aligned}
		&\frac{\partial f}{\partial t}-\nu\,\frac{\partial f}{\partial x}=-\mu f + \mu g\\
		&\frac{\partial g}{\partial t}+\nu\,\frac{\partial g}{\partial x}=\mu f - \mu g
	\end{aligned}\right.
\end{equation*}
where $(x,t)\in\R\times(0,\infty)$ and $\nu, \mu$ are positive constants.
Variables $f$ and $g$ represent ``densities'' of individuals moving, respectively, 
toward the left and toward the right of a one-dimensional line with velocity $\nu$.
The linear term at the right-hand side describes the fact that reversal of speed
is possible with a transition rate $\mu$.
For such a reason, the model is considered a differential description of a {\bf velocity-jump process}.
The Goldstein--Kac model was originally motivated by G.I. Taylor \cite{Tayl22} in what 
M. Kac describes as ``an abortive, or at least not very successful, attempt to treat turbulent
diffusion.'', see \cite{Kac74}.
Even if perhaps unproductive in its original intent, the model gained a lot of attention because
of its quality of being located at the crossroads of amenability and significance.
The former emerges from its hyperbolic and linear structure, giving raise to well-posedness 
and preservation of smoothness, and its quasi-monotonicity, guaranteeing the validity of
a comparison principle (see \cite{NataHano96}).
``Significance'' stems mainly from the fact that the Goldstein--Kac is considered as
a prototype for a differential description of transport mechanisms and, because of its nature,
a paradigm for {\bf hyperbolic diffusion}.
Such fact is even more evident observing that the sum $u=f+g$ satisfies the
{\bf telegraph equation} (or {\bf damped wave equation} in one space dimension),
\begin{equation*}
	2\mu\frac{\partial u}{\partial t}+\frac{\partial^2 u}{\partial t^2}
		-\nu^2 \frac{\partial^2 u}{\partial x^2}=0,
\end{equation*} 
often believed as a diffusion model with finite propagation speed, alternative to the traditional
parabolic heat equation, which, on the contrary, mantains the inherent paradox of infinite speed
of spreading (among others, see the interesting contribution \cite{Kell04}).

The aim of this article is to analyze a class of first order linear hyperbolic systems in several
space dimensions extending the Goldstein--Kac model, and still preserving the basic target of
describing a process where individuals may change velocity of propagation at a given rate. 
The main topic under investigation is to show how and at which extent the corresponding
hyperbolic diffusion mechanism is related to a corresponding parabolic one.
To enter the details, let $\R^d$ be the ambient space for the space variable $\mathbf{x}$
and let a fixed finite set of speeds $\{\mathbf{v}^1,\dots,\mathbf{v}^n\}\subset\R^d$ be given.
Finally, in order to describe the possibility of changing from one velocity to the other,
let $\mu_{ij}$, for $i,j\in\{1,\dots,n\}$ and $i\neq j$, be a set of non-negative constants measuring 
the rate of transition from speed $\mathbf{v}^i$ to speed $\mathbf{v}^j$.
Then, we consider the system for the unknown $\mathbf{f}=(f_1,\dots,f_n)$ given by
\begin{equation}\label{discrkin0}
  	\frac{\partial f_i}{\partial t}+\mathbf{v}^i\cdot\nabla_{\mathbf{x}} f_i
  		+\sum_{j\neq i}\bigl(\mu_{ij}\,f_i-\mu_{ji}\,f_j\bigr)=0, \qquad i=1,\dots,n,
\end{equation}
The Goldstein--Kac model corresponds to the one-dimensional case with two speeds,
namely $d=1$, $n=2$, $\{\mathbf{v}^1=-\nu,\mathbf{v}^2=+\nu\}\subset\R^1$, and with $\mu_{12}=\mu_{21}=\mu>0$.
Model \eqref{discrkin0} can also be regarded as a discrete velocity version of the model
treated in \cite{HillOthm00} or as a linearization of a discrete velocity Boltzmann system.
In this latter connection, system \eqref{discrkin0} fits into the class considered in \cite{ShizKawa85}.

Similarly to the destiny of the Goldstein--Kac model, system \eqref{discrkin0} can be considered
as a backbone for more complicated models taking into account the dependencies on
external signals, as in the case of chemotaxis (see \cite{HillStev00, GuMaNaRi09}),
or the presence of reaction/reproduction mechanisms (see \cite{Hade94, LaMaPlSi13}).
Endorsing the interest of such kind of generalizations, one-dimensional versions
of \eqref{discrkin0} has already been considered in \cite{FrieCrac06, FrieHuKeen13}
in the description of transport along axons.

In what follows, the model \eqref{discrkin0} is examined under the assumption that the 
transition rates are symmetric, $\mu_{ij}=\mu_{ji}$, the meaning of which in term of the 
velocity-jump process is immediate.
The non-symmetric case is still very interesting and natural in the modeling of phenomena
where the transport is guided by the gradient of some substance (as in the chemotaxis models)
and it is left for future investigation.
No specific restriction is made on the number and on the choice of the propagation
velocities $\mathbf{v}^i$.

In Section \ref{sect:velo}, the basic properties of \eqref{discrkin0} are discussed.
In particular, a brief description on how the model can be formally derived starting from
a correlated random walk is given, following the lines of what it is usually done for
the Goldstein--Kac model.
Then, it is showed that any convex function from $\R$ to $\R$ can be used to define
a Lyapunov functional for the system.
Such demanding structure implies that $L^p-$norm are preserved and that a form
of comparison principle holds.

Section \ref{sect:drift} can be considered the hub of the paper.
Applying the usual Laplace--Fourier transform, the hyperbolic system is converted into its
{\bf dispersion relation}, which encodes the different modes supported by the model.
In particular, the analysis of the dispersion relation close to the origin in the 
frequency space, is supposed to describe the large time behavior of the solutions.
A detailed analysis provides two crucial informations which are described in a somewhat
informal way here.
Comprehensive presentation and rigorous statements can be found in Section \ref{sect:drift}.

Firstly, under the assumption of symmetry of the transition rates, the large time behavior
of the variable $u=\sum_{i} f_i$ senses a drift speed $\mathbf{v}_{{}_{\textrm{drift}}}$
that is the arithmetic average of the speeds $\mathbf{v}^i$, viz.
\begin{equation*}
	\mathbf{v}_{{}_{\textrm{drift}}}=\frac{1}{n}\sum_{i=1}^n \mathbf{v}^i
\end{equation*}
(see Theorem \ref{thm:drift}).
In the non-symmetric case, the drift velocity does not coincide with the average
of the speeds, in general.

Secondly, in a frame moving with speed $\mathbf{v}_{{}_{\textrm{drift}}}$, the dynamics of the 
cumulative variable $u$ is related to the one of a linear constant-coefficient parabolic
diffusion equation determined by a real $d\times d$ matrix $\mathbb{D}$
\begin{equation}\label{parabeq}
	\frac{\partial w}{\partial t}=\textrm{\normalfont div}\left(\mathbb{D}\nabla_{\mathbf{x}} w\right)
\end{equation} 
The {\bf diffusion matrix} $\mathbb{D}$ is symmetric, non-negative definite;
additionally, it is possible to provide an explicit formula for the matrix $\mathbb{D}$, based on an
associated (undirected) graph, whose vertices are the velocities $\mathbf{v}^i$ and
whose arcs are weighted by the transition rates $\mu_{ij}$
(see Theorem \ref{thm:finaldiff}).
Such representation formula is based in a fundamental way on a variant of a well-known result
in graph theory, the so-called {\bf Kirchoff's matrix tree Theorem}.

Finally, Section \ref{sect:asym} is devoted to prove a rigorous result supporting the fact that
the parabolic equation \eqref{parabeq} gives the asymptotic description of \eqref{discrkin0}.
Specifically, chosen an initial datum $\mathbf{f}_0\in [L^1\cap L^2(\R^d)]^n$ for \eqref{discrkin0}
and denoted by $u$ the cumulative variable $\sum_i f_i$ and by $u_{{}_{\textrm{par}}}$
the solution to \eqref{parabeq} with initial datum $u_0(x):=\sum_{i} f_{0,i}$, there holds
\begin{equation*}
	|u-u_{{}_{\textrm{par}}}|_{{}_{L^2}}(t)
		\leq C\,t^{-\frac{1}{4}d-\frac{1}{2}}|{\mathbf{f}}_0|_{{}_{L^1\cap L^2}}
\end{equation*}
(see Theorem \ref{thm:asymptotic}).
The rate of the $L^2-$decay for \eqref{parabeq} with data in $L^1\cap L^2$ is $t^{-d/4}$,
thus the above estimate shows that the hyperbolic variable $u$ and its parabolic counterpart
$u_{{}_{\textrm{par}}}$ get closer one to the other in a time-scale shorter than the one of their
ultimate decay to zero.
This kind of estimate fits into a wide research stream exploring asymptotically parabolic nature of
hyperbolic equations, which dates back at least to J. Hadamard, \cite{Hada23}.
Additional bibliographical description on the subject is given in Section \ref{sect:asym}.

\section{Velocity-jump processes with a finite number of speeds}\label{sect:velo}

Let $d\geq 1$ and consider a family of $n$ velocities $\mathbf{v}^1,\dots,\mathbf{v}^n\in\R^d$,
with components $\mathbf{v}^i=(v_j^i)$, together with parameters $\mu_{ij}\geq 0$ for
$i,j=1,\dots, n$, $i\neq j$, describing the transition rate from speed $\mathbf{v}^i$ to
speed $\mathbf{v}^j$.
In all of the paper, the {\bf transition rates} $\mu_{ij}$ are assumed to be symmetric, i.e.
\begin{equation}\label{symmtrans}
	\mu_{ij}=\mu_{ji}\qquad\qquad\forall i\neq j.
\end{equation}
Given a population with total density $u=u(x,t)$, all the individuals are allowed
to move with one of the speeds $\mathbf{v}^1,\dots,\mathbf{v}^n\in\R^d$.
Denoted by $f_i=f_i(x,t)$ the density for the portion of the total population proceeding with velocity
$\mathbf{v}^i$ and assuming that the speed change is described by the rates $\mu_{ij}$, 
the dynamics is dictated by the first order linear system of hyperbolic type
\begin{equation}\label{discrkin}
  	\frac{\partial f_i}{\partial t}+\mathbf{v}^i\cdot\nabla_{\mathbf{x}} f_i
  		+\sum_{j\neq i}\bigl(\mu_{ij}\,f_i-\mu_{ji}\,f_j\bigr)=0, \qquad i=1,\dots,n,
\end{equation}
or, in vector form, with $\mathbf{f}=(f_1,\dots,f_n)$,
\begin{equation*}
 	 \frac{\partial \mathbf{f}}{\partial t}
	 	+\sum_{j=1}^{d} \mathbb{A}_j \frac{\partial \mathbf{f}}{\partial x_j}
  		+\mathbb{B}\mathbf{f}=0
\end{equation*}
where $\mathbb{A}_j=\diag(v_j^i)$ and $\mathbb{B}=(-\mu_{ij})$ with $\mu_{ii}:=-\sum\limits_{j\neq i} \mu_{ji}$.
If the coefficient $\mu_{ij}$ is zero for some $i,j$, then there is no direct transition from the speed
$\mathbf{v}^i$ to speed $\mathbf{v}^j$.

Matrix $\mathbb{B}$, referred to as the {\bf transition matrix}, is symmetric by assumption \eqref{symmtrans}
and singular since the sum of its columns is zero.
The total density $u=\sum f_i$ satisfies the homogeneous transport equation
\begin{equation*}
	\frac{\partial u}{\partial t}+\dive \mathbf{j}=0,
\end{equation*}
where the flux $\mathbf{j}$ is $\sum_{i} f_i\mathbf{v}^i$.

The Cauchy problem for \eqref{discrkin} determined by the initial condition
\begin{equation}\label{initialdatum}
 f_i(\mathbf{x},0)=f_{0,i}(\mathbf{x})\qquad\qquad \mathbf{x}\in\mathbb{R}^d,\quad i=1,\dots,n
\end{equation}
has a unique (mild) solution continuously dependent on the initial data whenever the initial datum
$\mathbf{f}_0=(f_{0,i})$ is chosen in an appropriate functional space.
Later on, we will concentrate on the case $f_0\in [L^1\cap L^2(\R^d)]^n$; for the moment,
we continue the discussion with choices of initial data depending case by case.

\subsection*{Derivation from a correlated random walk}
System \eqref{discrkin} can be heuristically derived from a correlated random walk,
in the same spirit of what is usually done for the Goldstein--Kac model.
Given the velocities $\{\mathbf{v}^1,\dots,\mathbf{v}^n\}\subset\R^d$ and the transition rates
$\mu_{ij}$, with $i,j\in\{1,\dots,n\}$, $i\neq j$, let $\dt>0$ be such that
\begin{equation*}
	p_i:=1-\sum_{j\neq i}\mu_{ij}\dt\geq 0\qquad\textrm{for any }i=1,\dots,n.
\end{equation*}
Then, let $X$ be the set of points in $\R^d$, defined by
\begin{equation*}
	X:=\Bigl\{\mathbf{x}=\sum_{i=1}^n c_i\,\mathbf{v}^i\dt\,:\,c_1,\dots,c_n\in\N\Bigr\}.
\end{equation*}
Assume that each particle of a given finite set is located at the initial time $t=0$ at some $\mathbf{x}\in X$
and it has a given state $i\in\{1,\dots,n\}$, corresponding to a ``preferential''
speed. The set of particles with state $i$ will be denoted by $F_i$.
At each time interval $\dt$, the displacement of every particle in $F_i$ amounts to $\dx=\mathbf{v}^i\dt$
with a probability $p_i$ and to $\dx=\mathbf{v}^j\dt$ with a probability $\mu_{ij}\dt$.
In the latter case, the particle changes state from $i$ to $j$.

Denoting by $f_i(\mathbf{x},n\,\dt)$ the fraction of particles with state $i$ that at time $t=n\,\dt$
are at position $\mathbf{x}$, the relation between the values $f_i$ at step $n$ and $n+1$ is
\begin{equation*}
	f_i(\mathbf{x},(n+1)\dt)=p_i f_i(\mathbf{x}-\mathbf{v}^i\dt,n\,dt)
		+\sum_{j\neq i} \mu_{ji} f_j(\mathbf{x}-\mathbf{v}^i\dt,n\,\dt)\dt
\end{equation*}
Adding and subtracting the term $f_i(\mathbf{x}+\mathbf{v}^i\dt,n)$, we obtain
\begin{equation*}
	\begin{aligned}
		&f_i(\mathbf{x}+\mathbf{v}^i\dt,(n+1)\dt)-f_i(\mathbf{x}+\mathbf{v}^i\dt,n\,\dt)
			+f_i(\mathbf{x}+\mathbf{v}^i\dt,n\,\dt)-f_i(\mathbf{x},n\,\dt)\\
		&\hskip5.5cm =\Bigl(\sum_{j\neq i}\mu_{ij}\Bigr) f_i(\mathbf{x},n\,\dt)\dt
			+\sum_{j\neq i} \mu_{ji} f_j(\mathbf{x},n\,\dt)\dt
	\end{aligned}
\end{equation*}
For time interval $\dt$ small, assuming $f_i$ to be smooth with respect to its first argument,
we may approximate the difference $f_i(\mathbf{x}+\mathbf{v}^i\dt,n\,\dt)-f_i(\mathbf{x},n\,\dt)$
with the scalar product between the gradient of $f$ with respect to $\mathbf{x}$ and the increment
$\mathbf{v}^i\dt$, getting the relation
\begin{equation*}
	\frac{1}{\dt}\bigl\{f_i(\mathbf{x}+\mathbf{v}^i\dt,(n+1)\dt)-f_i(\mathbf{x}+\mathbf{v}^i\dt,n\,\dt)\bigr\}
		+\mathbf{v}^i\cdot\nabla_{\mathbf{x}} f_i+\sum_{j\neq i}\bigl(\mu_{ij} f_i-\mu_{ji} f_j\bigr)=0.
\end{equation*}
Passing to the limit $\dt\to 0$, we formally obtain \eqref{discrkin}.

\subsection*{Properties of the first order linear system.}
The special structure of the zero-th order term in \eqref{discrkin} triggers a number of 
additional properties for the solutions to the system.

\begin{proposition}\label{prop:decayeta}
Assume \eqref{symmtrans}.
Let $\eta$ be a Lipschitz continuous and convex function from $\R$ to $\R$.
Then, each solution $\mathbf{f}$ to \eqref{discrkin} is such that, whenever the
right-hand side is finite,
\begin{equation}\label{decayeta}
	\sum_{i=1}^{n} \int_{\R^d} \eta(f_i)(\mathbf{x},t_2)\,d\mathbf{x}
		\leq \sum_{i=1}^{n} \int_{\R^d} \eta(f_{i})(\mathbf{x},t_1)\,d\mathbf{x}
\end{equation}
for any $t_1<t_2$.
\end{proposition}

\begin{proof}
Let us consider the case of a smooth initial datum $\mathbf{f}_0$ so that
the solution $\mathbf{f}$ is also smooth.
The general case can be obtained by applying a density argument.

Given a Lipschitz continuous function $\eta$, 
multiplying \eqref{discrkin} by $\eta'(f_i)$ and summing up with respect to $i$,
we infer
\begin{equation*}
	\frac{\partial }{\partial t}\sum_{i=1}^n \eta(f_i)
		+\sum_{\ell=1}^d \frac{\partial}{\partial x_\ell}\left(\sum_{i=1}^n  v_i^\ell\,\eta(f_i)\right)
  		+\sum_{i=1}^n \sum_{j\neq i} \eta'(f_i) \bigl(\mu_{ji}\,f_i-\mu_{ij}\,f_j\bigr)=0,
\end{equation*}
Given $t_1<t_2$, integrating with respect to $(x,t)$ in $\R^d\times [t_1,t_2]$, we get
\begin{equation*}
	\sum_{i=1}^n \int_{\R^d}  \eta(f_i)(\mathbf{x},t_2)\,d\mathbf{x}
		+\int_{\R^d}\mathcal{I}[\mathbf{f}]\,d\mathbf{x}\,dt
		=\sum_{i=1}^n\int_{\R^d}  \eta(f_i)(\mathbf{x},t_1)\,d\mathbf{x},
\end{equation*}
where the function $\mathcal{I}$ is defined by
\begin{equation*}
	\mathcal{I}[\mathbf{f}]:=\sum_{i=1}^n \sum_{j\neq i} \eta'(f_i) \bigl(\mu_{ji}\,f_i-\mu_{ij}\,f_j\bigr).
\end{equation*}
Since $\mu_{ij}=\mu_{ji}$,  there holds
\begin{equation}\label{Iforms}
	\begin{aligned}
	\mathcal{I}[\mathbf{f}]&=\sum_{i, j=1}^n \eta'(f_i) \bigl(\mu_{ji}\,f_i-\mu_{ij}\,f_j\bigr)
	=\sum_{i,j=1}^n \mu_{ji}\left(\eta'(f_i)-\eta'(f_j)\right)f_i\\
	&=\frac12\sum_{i,j=1}^n \mu_{ji}\left(\eta'(f_i)-\eta'(f_j)\right)f_i
		+\frac12\sum_{i,j=1}^n \mu_{ij}\left(\eta'(f_j)-\eta'(f_i)\right)f_j\\
	&=\frac12\sum_{i,j=1}^n \mu_{ij}\left(\eta'(f_i)-\eta'(f_j)\right)(f_i-f_j)\geq 0,
	\end{aligned}
\end{equation}
for $\eta'$ increasing and $\mu_{ij}\geq 0$.
\end{proof}

In particular, the solution semigroup of \eqref{discrkin} is such that the $L^p-$norms
are not increasing in time for any $p\geq 1$.
Additionally, also a comparison principle holds as a consequence
of Proposition \ref{prop:decayeta}.

\begin{corollary}
Let $\mathbf{f}$ and $\mathbf{g}$ be two solutions to \eqref{discrkin} corresponding
to the initial conditions $\mathbf{f}(\mathbf{x},0)=\mathbf{f}_0(\mathbf{x})$ and
$\mathbf{g}(\mathbf{x},0)=\mathbf{g}_0(\mathbf{x})$, respectively.
If the initial data $\mathbf{f}_0$ and $\mathbf{g}_0$ are such that
\begin{equation*}
	f_{0,i}\leq g_{0,i}\qquad\quad\forall i=1,\dots,n
\end{equation*}
then the same ordering relation holds for any positive time, i.e.
\begin{equation*}
	f_{i}(\mathbf{x},t)\leq g_{i}(\mathbf{x},t)\qquad\quad\forall i=1,\dots,n
\end{equation*}
for any $t>0$.
\end{corollary}

\begin{proof} The linearity of the equation permits to restrict the attention to the case $g_0=0$.
Choosing $\eta(s)=[s]_+$ in \eqref{decayeta}, we deduce
\begin{equation*}
	\sum_{i=1}^n \int_{\mathbb{R}^d}  [f_i(\mathbf{x},t)]_+\,d\mathbf{x}
	\leq \sum_{i=1}^n \int_{\mathbb{R}^d}  [f_i(\mathbf{x},0)]_+\,d\mathbf{x}.
\end{equation*}
In particular, if $f_i\leq 0$ for any $i$ at time $t=0$, then $f_i\leq 0$ for any $t\geq 0$ for any $i$.
\end{proof}

Property \eqref{decayeta} indicates a form a weak dissipation.
If the transition matrix $\mathbb{B}$ is irreducible (see \cite{BermPlem79}, Chap.2, Sect.2)
and an appropriate additional assumption on the velocities $\mathbf{v}^i$ holds, the
{\bf Shizuta--Kawashima dissipativity condition} is valid (see \cite{ShizKawa85}
and, for a discussion on its limits of validity, see \cite{MascNata10}).

\begin{proposition}\label{prop:sk}
Let the matrix $\mathbb{B}$ be irreducible and assume 
\begin{equation}\label{linspan}
	\linspan\{\mathbf{v}^i-\mathbf{v}^j\,:\, i,j=1,\dots,n\}=\R^d
\end{equation}
Then for any $F\in\ker\mathbb{B}$ there holds
\begin{equation*}
	\lambda\,F+\sum_{j=1}^n k_j \mathbb{A}_j F\neq 0
\end{equation*}
for any $\lambda\in\R$ and any $\mathbf{k}=(k_1,\dots,k_d)$ with $\mathbf{k}\neq 0$.
\end{proposition}

\begin{proof}
The specific structure of the matrix $\mathbb{B}$ and its irreducibility together imply
that $\ker\mathbb{B}$ is one-dimensional and generated by the vector $\mathbf{1}=(1,\dots,1)$.
Since $\bigl(\mathbb{A}_j \mathbf{1}\bigr)_i=v^i_j$, we infer
\begin{equation*}
	\lambda+\sum_{j=1}^n k_j \bigl(\mathbb{A}_j \mathbf{1}\bigr)_i
		=\lambda+\sum_{j=1}^n k_j v^i_j
		=\lambda+\mathbf{k}\cdot \mathbf{v}^i
	\qquad\qquad i=1,\dots,n.
\end{equation*}
If $\lambda+\mathbf{k}\cdot \mathbf{v}^i=0$ is null for any $i$, then
\begin{equation*}
	\mathbf{k}\cdot (\mathbf{v}^i-\mathbf{v}^j)=0
	\qquad\qquad\forall i,j
\end{equation*}
and thus $\mathbf{k}=0$ as a consequence of \eqref{linspan}.
\end{proof}

The role of the two hypotheses in Proposition \ref{prop:sk} can be clarified
showing that under these assumptions, invoking the {\bf LaSalle Invariance Principle},
the solution to the Cauchy problem \eqref{discrkin}--\eqref{initialdatum}
decayes to zero as $t\to+\infty$ for integrable initial datum in some $L^p$.
Indeed, the $\omega-$limit of a given trajectory is contained in the largest invariant region
for which the functional $\sum \int \eta(f_i)$ is constant.
For $\eta'$ strictly increasing, the final representation of $\mathcal{I}[\mathbf{f}]$
in \eqref{Iforms} together with the requirement of the irreducibility of the matrix $\mathbb{B}$,
shows that
\begin{equation*}
	\mathcal{I}[\mathbf{f}]=0\qquad\iff\qquad f_i=f_j\quad\forall\,i,j.
\end{equation*}
Hence, solutions preserving the value of $\sum \int \eta(f_i)$ are such that
$f_i=g$ for some scalar function $g$ satisfying the homogeneous system
\begin{equation*}
  \frac{\partial g}{\partial t}+\mathbf{v}^i\cdot\nabla_x g=0, \qquad i=1,\dots,n.
\end{equation*}
which gives
\begin{equation*}
  (\mathbf{v}^i-\mathbf{v}^j)\cdot\nabla_{\mathbf{x}} g=0, \qquad i,j=1,\dots,n.
\end{equation*}
Thus, as a consequence of \eqref{linspan}, $g$ is constant.

A rigorous application of the principle requires some compactness of the trajectories,
guaranteed by Sobolev estimates at the price of some additional regularity requirements
on the initial data.
Indeed, taking again advantage of the dissipation described by Proposition \ref{prop:decayeta}
and the linear structure of the equation, one easily obtains
\begin{equation*}
	\sum_{i=1}^{n} |f_i(\cdot,t)|_{{}_{W^{k,p}}}^p
		\leq \sum_{i=1}^{n} |f_{0,i}(\cdot,t)|_{{}_{W^{k,p}}}^p
\end{equation*}
for any $k\geq 0, p\geq 1$.

All the above arguments concur in showing that the system \eqref{discrkin} shares
a number of properties with scalar linear diffusion equations.
Our next aim is to show that the connection between the two classes is much stronger than that
and that it is possible to identify a specific linear diffusion equation of paraboic type describing
the large-time behavior of the solution of the hyperbolic system \eqref{discrkin}.

\section{Drift velocity and diffusion matrix}\label{sect:drift}

Given a square matrix $\mathbb{A}$ of dimension $n\times n$ and $k$ indeces $i_1,\dots,i_k$, 
let $\mathbb{A}(i_1,\dots,i_k)$ be the principal minor that results
from deleting sets of $k$ rows and columns with indeces $i_1,\dots,i_k$.
By definition, we set $\mathbb{A}(1,\dots,n):=1$.
Morevor, given $n$ column vectors $w_1,\dots, w_n$ in $\R^n$, let 
\begin{equation*}
	w_1\land \dots \land w_n:=\det(w_1\dots w_n).
\end{equation*}
Applying Laplace--Fourier transform, that is converting time derivatives with multiplication
by $\lambda\in\R$ and space derivatives with scalar multiplication by $\mathbf{k}\in i\R^d$, 
\begin{equation*}
	(\partial_t,\nabla_{\mathbf{x}})\quad \mapsto \quad (\lambda,\mathbf{k})\cdot
\end{equation*}
the partial differential equation \eqref{discrkin} is converted into the linear system
\begin{equation*}
 (\lambda+\mathbf{v}^i\cdot \mathbf{k})\hat f_i+\sum_{j=1}^{n}\mu_{ij}\,\hat f_j=0
 \qquad\qquad i=1,\dots,n,
\end{equation*}
where $\hat f_i=\hat f_i(\lambda,\mathbf{k})$ is the transform of $f$.
Hence, the {\bf dispersion relation} of the system \eqref{discrkin} is
the polynomial relation in $\lambda$ and $\mathbf{k}$
\begin{equation*}
	p(\lambda,\mathbf{k}):=\det(\lambda\,\mathbb{I}+\diag(\mathbf{v}^i\cdot\mathbf{k})+\mathbb{B})=0.
\end{equation*}
Denoting by $B_i$ the columns of the matrix $\mathbb{B}$, the dispersion relation can be
written in compact form as
\begin{equation*}
	p(\lambda,\mathbf{k})=V_1(\lambda,\mathbf{k}) \land \cdots \land V_n(\lambda,\mathbf{k})=0.
\end{equation*}
where $V^i(\lambda,\mathbf{k}):=\bigl(\lambda+\mathbf{v}^i\cdot\mathbf{k}\bigr)E_i+B_i$
and $E_i=(\delta_{ij})_{j=1,\dots,n}$ is the $i$-th element of the canonical basis of $\R^n$.

Given $\mathbf{k}$, the main target of the analysis is to determine the location of the values
$\lambda=\lambda(\mathbf{k})$ such that the dispersion relation $p(\lambda,\mathbf{k})=0$ is satisfied.
In particular, being interested in the large-time behavior of the solutions, the attention is mainly 
directed to the region of $\mathbf{k}$ with $|\mathbf{k}|$ small.

\begin{proposition}\label{prop:expansion00}
Let $\mathbb{B}$ be a real $n\times n$ singular matrix such that
\begin{equation*}
	I_1(\mathbb{B}):=\sum_{i=1}^{n} \det\mathbb{B}(i)\neq 0.
\end{equation*}
Then, there is a smooth function $\mathbf{k}\mapsto \lambda(\mathbf{k})$ defined in a neighborhood
of $\mathbf{0}$ such that $p(\lambda,\mathbf{k})=0$ if and only if $\lambda=\lambda(\mathbf{k})$.
Moreover, there hold
\begin{equation*}
	\begin{aligned}
	&\textrm{first order:}		& \qquad	&
		\frac{\partial\lambda}{\partial k_\ell}(\mathbf{0})
			=-\frac{1}{I_1(\mathbb{B})}\sum_{i=1}^{n} v_\ell^i \det\mathbb{B}(i),\\
	&\textrm{second order:}	& \qquad	&
		\frac{\partial^2 \lambda}{\partial k_j\partial k_\ell}(\mathbf{0})
			=-\frac{1}{I_1(\mathbb{B})}\sum_{\{h,k\,:\,h\neq k\}} (v_h^j\,v_k^\ell)\,\mathbb{B}(h,k)
		\quad\textrm{if}\;\nabla_{\mathbf{k}}\lambda(\mathbf{0})=\mathbf{0}.
	\end{aligned}
\end{equation*}
\end{proposition}

\begin{proof}
Since $\det \mathbb{B}=0$, the couple $(0,\mathbf{0})$ satisfies the dispersion relation.
Since
\begin{equation*}
	\begin{aligned}
 	\frac{\partial p}{\partial \lambda}(\lambda,\mathbf{k})
  		&=\frac{\partial V_1}{\partial \lambda} \land V_2\land  \cdots \land V_n+\dots
   			+V_1 \land \cdots \land V_{n-1}\land \frac{\partial V_n}{\partial\lambda}\\
  		&=E_1 \land V_2\land \cdots \land V_n+\dots+V_1 \land \cdots  \land V_{n-1} \land E_n,
	\end{aligned}
\end{equation*}
there holds
\begin{equation*}
 	\frac{\partial p}{\partial \lambda}(0,\mathbf{0})
	 	=E_1 \land B_2\land \cdots \land B_n+\dots+B_1 \land \cdots  \land B_{n-1} \land E_n
  		=I_1(\mathbb{B})
\end{equation*}
and the existence of the function $\lambda$ follows from the Implicit function Theorem.

Moreover, differentiating with respect to $k_\ell$ for $\ell\in\{1,\dots,n\}$, we obtain. 
\begin{equation}\label{firstderivative}
	\begin{aligned}
 	\frac{\partial p}{\partial k_\ell}(\lambda,\mathbf{k})
  		&=\frac{\partial V_1}{\partial k_\ell} \land V_2 \land \cdots \land V_n+\dots
   			+V_1 \land \cdots \land V_{n-1} \land \frac{\partial V_n}{\partial k_\ell}\\
  		&=v_\ell^1 E_1\land V_2 \land \cdots \land V_n+\dots
   			+V_1 \land \cdots \land V_{n-1} \land v_\ell^n E_n;
	\end{aligned}
\end{equation}
and thus, calculating at $(0,\mathbf{0})$, we infer
\begin{equation*}
	\begin{aligned}
 	\frac{\partial p}{\partial k_\ell}(0,\mathbf{0})
  		&=v_\ell^1 E_1\land B_2 \land \cdots \land B_n+\dots
   			+B_1 \land \cdots \land B_{n-1} \land v_\ell^n E_n\\
		&=\sum_{i=1}^n v_\ell^i \det \mathbb{B}(i)
	\end{aligned}
\end{equation*}
which gives the first order expansion.

Differentating with respect to $k_\ell$ and then with respect to $k_j$ the relation $p(\lambda(k),k)=0$,
we get the equality
\begin{equation*}
	\frac{\partial^2 p}{\partial\lambda^2}
		\frac{\partial \lambda}{\partial k_\ell}\frac{\partial \lambda}{\partial k_j}
		+\frac{\partial^2 p}{\partial\lambda\partial k_j}\frac{\partial \lambda}{\partial k_\ell}
		+\frac{\partial p}{\partial\lambda}\frac{\partial^2 \lambda}{\partial k_j\partial k_\ell}
		+\frac{\partial^2 p}{\partial\lambda\partial k_\ell}\frac{\partial \lambda}{\partial k_j}
		+\frac{\partial^2 p}{\partial k_j\partial k_\ell}=0.
\end{equation*}
Calculating at $(0,\mathbf{0})$, since $\partial_\lambda p(0,\mathbf{0})=I_1(B)$ and
$\nabla_{\mathbf k} \lambda(\mathbf{0})=\mathbf{0}$,
we obtain
\begin{equation*}
	\frac{\partial^2 \lambda}{\partial\kappa_j\partial\kappa_\ell}(0,\mathbf{0})
	=-\frac{1}{I_1(\mathbb{B})}\frac{\partial^2 p}{\partial\kappa_j\partial\kappa_\ell}(0,\mathbf{0}).
\end{equation*}
Upon differentiation of \eqref{firstderivative}, we deduce
\begin{equation*}
	\begin{aligned}
 	\frac{\partial^2 p}{\partial k_j\,\partial k_\ell}(\lambda, \mathbf{k})
  	&=\frac{\partial}{\partial k_j}\left(v_1^\ell E_1\land V_2 \land \cdots \land V_n\right)
  			+\dots+\frac{\partial}{\partial k_j}\left(V_1 \land \cdots \land V_{n-1} \land v_n^\ell E_n\right)\\
  	&=\sum_{\{h,k\,:\,h\neq k\}} (v_h^j\,v_k^\ell)\,
   	V_1\land\cdots \land E_h \land \cdots \land E_k \land \cdots \land V_n.
 	 \end{aligned} 
\end{equation*}
Therefore, calculating at $(0,\mathbf{0})$, we end up with
\begin{equation*}
	\begin{aligned}
 	\frac{\partial^2 p}{\partial\kappa_j\,\partial\kappa_\ell}(0,\mathbf{0})
 		 &=\sum_{\{h,k\,:\,h\neq k\}} (v_h^j\,v_k^\ell)\,
  			B_1\land\cdots \land E_h \land \cdots \land E_k \land \cdots \land B_n\\
  		&=\sum_{\{h,k\,:\,h\neq k\}} (v_h^j\,v_k^\ell)\,\mathbb{B}(h,k),
	\end{aligned}
\end{equation*}
that gives the second order derivatives of $\lambda=\lambda(\mathbf{k})$ in the origin.
\end{proof}

Condition $I_1(\mathbb{B})\neq 0$ is satisfied if the matrix $\mathbb{B}$ is irreducible.
From now on, we will consider such assumption to be satisfied.

The expressions for the first and second order term in the expansion of the function $\lambda$ 
at $\mathbf{k}=0$ can be restyled.
The gradient of $\lambda$ can be rewritten as
\begin{equation*}
	\nabla_{\mathbf{k}} \lambda(\mathbf{0})=-\frac{1}{I_1(\mathbb{B})}
		\bigl((\mathbb{B}(1),\dots,\mathbb{B}(n))\cdot(v^1_j,\dots,v^n_j)\bigr)_{j=1,\dots,d}
\end{equation*}
Also, when the gradient of $\lambda$ at $\mathbf{k}=\mathbf{0}$ is null,
the hessian matrix of $\lambda$ at the same point is 
\begin{equation*}
	D^2 \lambda(\mathbf{0})=-\frac{2}{I_1(\mathbb{B})}
		\sum_{h<k} \mathbb{B}(h,  k)\,(\mathbf{v}^h\otimes \mathbf{v}^k)^*
\end{equation*}
where $\otimes$ is the usual tensor product of vectors and ${}^*$ denotes the symmetric part of a matrix.

With a terminology that will be fully motivated in the subsequent Section, we define
the {\bf drift velocity} of the system \eqref{discrkin} to be the vector
\begin{equation}\label{drift}
	\mathbf{v}_{{}_{\textrm{drift}}}:=-\nabla_{\mathbf{k}} \lambda(\mathbf{0})
	=\frac{1}{I_1(\mathbb{B})}
		\bigl((\det\mathbb{B}(1),\dots,\det\mathbb{B}(n))\cdot(v^1_j,\dots,v^n_j)\bigr)_{j=1,\dots,d}
\end{equation}
and the {\bf diffusion matrix} of the system \eqref{discrkin} as
\begin{equation}\label{diffmatrix}
	\mathbb{D}:=D^2 \lambda(\mathbf{0})=-\frac{2}{I_1(\mathbb{B})}
		\sum_{i<j} \mathbb{B}(i,j)\,(\mathbf{v}^i\otimes \mathbf{v}^j)^*.
\end{equation}
In a frame moving with speed $\mathbf{v}_{{}_{\textrm{drift}}}$, the velocities $\mathbf{v}^i$
in the system are modified in $\mathbf{v}^i-\mathbf{v}_{{}_{\textrm{drift}}}$ and coherently, the 
gradient of $\lambda$ at $\mathbf{0}$ is null.
Thus, without loss of generality, we may assume that the conditions
\begin{equation*}
	(\det\mathbb{B}(1),\dots,\det\mathbb{B}(n))\cdot(v^1_j,\dots,v^n_j)=\mathbf{0}
	\qquad\qquad\forall j=1,\dots,d
\end{equation*}
henceforth holds.

\begin{example}\label{ex:goldsteinkac} [Goldstein--Kac model] \normalfont
The simplest example needs only two velocities and a transition matrix of the form
\begin{equation*}
 	\mathbb{B}=\begin{pmatrix}  \mu_{12} & -\mu_{21} \\  -\mu_{12} & \mu_{21} \end{pmatrix}.
\end{equation*}
for some positive values $\mu_{12}, \mu_{21}$.
Since $\mathbb{B}(1)=\mu_{21}$ and $\mathbb{B}(2)=\mu_{12}$, the drift velocity of the system is 
\begin{equation*}
	\mathbf{v}_{{}_{\textrm{drift}}}
		=\frac{\mu_{21}\mathbf{v}^1+\mu_{12}\mathbf{v^2}}{\mu_{21}+\mu_{12}}.
\end{equation*}
Since $\det\mathbb{B}(1,2)=1$, if $\mu_{21}\mathbf{v}^1+\mu_{12}\mathbf{v^2}=0$, there holds
\begin{equation*}
	\mathbb{D}=-\frac{2}{\mu_{21}+\mu_{12}}(\mathbf{v}^1\otimes \mathbf{v}^2)^*
		=\frac{2\mu_{21}/\mu_{12}}{\mu_{21}+\mu_{12}}(\mathbf{v}^1\otimes \mathbf{v}^1)
\end{equation*}
When the symmetry condition $\mu_{21}=\mu_{12}$ holds, $\mathbf{v}_{{}_{\textrm{drift}}}$
is the algebraic mean of $\mathbf{v}^1$ and $\mathbf{v}^2$ and the diffusion matrix $\mathbb{D}$
reduces to $(\mathbf{v}^1\otimes \mathbf{v}^1)/\mu_{12}$.
\end{example}

\begin{example}\label{ex:three} \normalfont
Going a step further, let us consider the symmetric case with $n=3$
given by the symmetric transition matrix
\begin{equation*}
 	\mathbb{B}=\left(\begin{array}{ccc}  
			b+c	& -c		& -b \\
			-c	& a+c	& -a \\
			-b	& -a		& a+b
			\end{array}\right).
\end{equation*}
so that
\begin{equation*}
	\begin{aligned}
		&\det \mathbb{B}(1)=\det \mathbb{B}(2)=\det \mathbb{B}(3)=ab+ac+bc,\\
		&\det \mathbb{B}(1,2)=a+b,\quad 
			\det \mathbb{B}(1,3)=a+c,\quad
				\det \mathbb{B}(2,3)=b+c.
	\end{aligned}
\end{equation*}
\begin{figure}[hbt] 
\begin{tikzpicture}[scale=1.75, auto, swap]
    \foreach \pos/\name in {{(1,0)/1}, {(-0.5,0.87)/2}, {(-0.5,-0.87)/3}}
        \node[vertex] (\name) at \pos {$\name$};
    \foreach \source/ \dest /\weight in {1/2/c, 2/3/a,3/1/b}
        \path[edge] (\source) -- node[weight] {$\weight$} (\dest);
 \end{tikzpicture}
\caption{\footnotesize Schematic represenation of the transitions in Examples \ref{ex:three}.}
\end{figure}
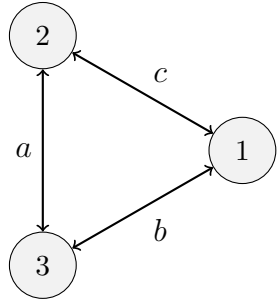
Therefore, the drift velocity is the average of $\mathbf{v}^1, \mathbf{v}^2, \mathbf{v}^3$
\begin{equation*}
	\mathbf{v}_{{}_{\textrm{drift}}}=\frac13\left(\mathbf{v}^1+\mathbf{v}^2+\mathbf{v}^3\right),
\end{equation*}
and, for $\mathbf{v}_{{}_{\textrm{drift}}}=0$, the diffusion matrix is
\begin{equation*}
	\begin{aligned}
	\mathbb{D}&=-\frac{2}{3(ab+bc+ca)}\bigl\{(a+b)(\mathbf{v}^1\otimes \mathbf{v}^2)^*
			+(a+c)(\mathbf{v}^1\otimes \mathbf{v}^3)^*
			+(b+c)(\mathbf{v}^2\otimes \mathbf{v}^3)^*\bigr\}\\
		&=-\frac{2}{3(ab+bc+ca)}\bigl\{
			a\bigl[(\mathbf{v}^1\otimes \mathbf{v}^2)^*+(\mathbf{v}^1\otimes \mathbf{v}^3)^*\bigr]\\
		&\hskip3.5cm
			+b\bigl[(\mathbf{v}^1\otimes \mathbf{v}^2)^*+(\mathbf{v}^2\otimes \mathbf{v}^3)^*\bigr]
			+c\bigl[(\mathbf{v}^1\otimes \mathbf{v}^3)^*+(\mathbf{v}^2\otimes \mathbf{v}^3)^*\bigr]\bigr\}\\
	&=\frac{2}{3(ab+bc+ca)}\left\{a\,\mathbf{v}^1\otimes \mathbf{v}^1
			+b\,\mathbf{v}^2\otimes \mathbf{v}^2+c\,\mathbf{v}^3\otimes \mathbf{v}^3\right\}
	\end{aligned}
\end{equation*}
\end{example}

\begin{example}\label{ex:fourside} \normalfont
Next, consider four velocities with symmetric transition rates
\begin{equation*}
 	\mathbb{B}=\left(\begin{array}{cccc}  
			\ast		& -\mu_{12}	& 0			& -\mu_{14} 	\\
			-\mu_{12}	& \ast           	& -\mu_{23} 	& 0 			\\
			0		& -\mu_{23}	& \ast           	& -\mu_{34} 	\\
			-\mu_{14} & 0			& -\mu_{34} 	&\ast	
			\end{array}\right),
\end{equation*}
where the diagonal entries are the sums of the column element changed by sign.
Then, there hold for $i\in\{1,2,3,4\}$
\begin{equation*}
	\det\mathbb{B}(i)=\mu_{12}\mu_{14}\mu_{23}+\mu_{12}\mu_{14}\mu_{34}
						+\mu_{12}\mu_{23}\mu_{34}+\mu_{14}\mu_{23}\mu_{34}
\end{equation*}
and, for the second order principal minors,
\begin{equation*}
	\begin{aligned}
		\det\mathbb{B}(1,2)&=\mu_{14}\mu_{23}+\mu_{14}\mu_{34}+\mu_{23}\mu_{34}\\
		\det\mathbb{B}(1,3)&=\mu_{12}\mu_{14}+\mu_{12}\mu_{34}+\mu_{14}\mu_{23}+\mu_{23}\mu_{34}\\
		\det\mathbb{B}(1,4)&=\mu_{12}\mu_{23}+\mu_{12}\mu_{34}+\mu_{23}\mu_{34}\\
		\det\mathbb{B}(2,3)&=\mu_{12}\mu_{14}+\mu_{12}\mu_{34}+\mu_{14}\mu_{34}\\
		\det\mathbb{B}(2,4)&=\mu_{12}\mu_{23}+\mu_{12}\mu_{34}+\mu_{14}\mu_{23}+\mu_{14}\mu_{34}\\
		\det\mathbb{B}(3,4)&=\mu_{12}\mu_{14}+\mu_{12}\mu_{23}+\mu_{14}\mu_{23}
	\end{aligned}
\end{equation*}
Since $\det\mathbb{B}(i)$ is independent of $i$, the drift velocity is the average of the speeds
\begin{equation*}
	\mathbf{v}_{{}_{\textrm{drift}}}=\frac{1}{4}\bigl(\mathbf{v}^1+\mathbf{v}^2+\mathbf{v}^3+\mathbf{v}^4\bigr).
\end{equation*}
Then, assuming $\mathbf{v}_{{}_{\textrm{drift}}}=0$, the expression for the diffusion matrix can be rearranged 
by collecting the term multiplied by the same product of transition rates, obtaining
\begin{equation*}
	\begin{aligned}
	\mathbb{D}&=-\frac{2}{\det\mathbb{B}(i)}\Bigl\{
			\mu_{23}\mu_{34}\bigl((\mathbf{v}^2+\mathbf{v}^3+\mathbf{v}^4)\otimes \mathbf{v}^1\bigr)^*
			+\mu_{14}\mu_{34}\bigl((\mathbf{v}^1+\mathbf{v}^3+\mathbf{v}^4)\otimes \mathbf{v}^2\bigr)^*\\
		&\qquad\qquad 
			+\mu_{12}\mu_{14}\bigl((\mathbf{v}^1+\mathbf{v}^2+\mathbf{v}^4)\otimes \mathbf{v}^3\bigr)^*
			+\mu_{12}\mu_{23}\bigl((\mathbf{v}^1+\mathbf{v}^2+\mathbf{v}^3)\otimes \mathbf{v}^4\bigr)^*\\
		&\qquad\qquad
			+\mu_{12}\mu_{34}\bigl((\mathbf{v}^1+\mathbf{v}^2)\otimes(\mathbf{v}^3+\mathbf{v}^4)\bigr)^*
			+\mu_{14}\mu_{23}\bigl((\mathbf{v}^1+\mathbf{v}^4)\otimes(\mathbf{v}^2+\mathbf{v}^3)\bigr)^*\Bigr\}.
	\end{aligned}
\end{equation*}
Since $\mathbf{v}^1+\mathbf{v}^2+\mathbf{v}^3+\mathbf{v}^4=0$, the matrix $\mathbb{D}$ can be rewritten as
\begin{equation*}
	\begin{aligned}
	\mathbb{D}&=\frac{2}{\det\mathbb{B}(i)}\Bigl\{
			\mu_{23}\mu_{34} (\mathbf{v}^1\otimes \mathbf{v}^1)
				+\mu_{14}\mu_{34} (\mathbf{v}^2\otimes \mathbf{v}^2)
					+\mu_{12}\mu_{14}(\mathbf{v}^3\otimes \mathbf{v}^3)
						+\mu_{12}\mu_{23}(\mathbf{v}^4\otimes \mathbf{v}^4)\\
		&\qquad	+\mu_{12}\mu_{34}\bigl((\mathbf{v}^1+\mathbf{v}^2)\otimes(\mathbf{v}^1+\mathbf{v}^2)\bigr)
			+\mu_{14}\mu_{23}\bigl((\mathbf{v}^1+\mathbf{v}^4)\otimes(\mathbf{v}^1+\mathbf{v}^4)\bigr)\Bigr\}.
	\end{aligned}
\end{equation*}
showing, in particular, that the diffusion matrix is non-negative definite.
\end{example}

\begin{example}\label{ex:fourtree} \normalfont
As a final example, let us consider a case with 5 velocities and admissible transitions
only between $\mathbf{v}^1$ and $\mathbf{v}^j$ for $j=2,3,4,5$.
The transition matrix is
\begin{equation*}
 	\mathbb{B}=\left(\begin{array}{ccccc}  
			\ast		& -\mu_{12}	& -\mu_{13}	& -\mu_{14} 	& -\mu_{15} 	\\
			-\mu_{12}	& \ast           	& 0			& 0 			& 0			\\
			-\mu_{13}	& 0			& \ast           	& 0		 	& 0			\\
			-\mu_{14} & 0			& 0			&\ast			& 0			\\
			-\mu_{15} & 0			& 0			&0			& \ast		\\
			\end{array}\right).
\end{equation*}
where the diagonal entries are the sums of the column element changed by sign.
A direct computation shows that $\det\mathbb{B}(i)=\mu_{12}\mu_{13}\mu_{14}\mu_{15}$
for any $i\in\{1,2,3,4,5\}$.
The second order principal minors are
\begin{equation*}
	\begin{aligned}
		\det\mathbb{B}(1,2)&=\mu_{13}\mu_{14}\mu_{15}
			&\qquad	 	\det\mathbb{B}(1,3)&=\mu_{12}\mu_{14}\mu_{15}\\
		\det\mathbb{B}(1,4)&=\mu_{12}\mu_{13}\mu_{15}
			&\qquad		\det\mathbb{B}(1,5)&=\mu_{12}\mu_{13}\mu_{14}\\
		\det\mathbb{B}(2,3)&=\mu_{12}\mu_{14}\mu_{15}+\mu_{13}\mu_{14}\mu_{15}
			&\quad	\det\mathbb{B}(2,4)&=\mu_{12}\mu_{13}\mu_{15}+\mu_{13}\mu_{14}\mu_{15}\\
		\det\mathbb{B}(2,5)&=\mu_{12}\mu_{13}\mu_{14}+\mu_{13}\mu_{14}\mu_{15}
			&\quad	\det\mathbb{B}(3,4)&=\mu_{12}\mu_{13}\mu_{15}+\mu_{12}\mu_{14}\mu_{15}\\
		\det\mathbb{B}(3,5)&=\mu_{12}\mu_{13}\mu_{14}+\mu_{12}\mu_{14}\mu_{15}
			&\quad	\det\mathbb{B}(4,5)&=\mu_{12}\mu_{13}\mu_{14}+\mu_{12}\mu_{13}\mu_{15}.
	\end{aligned}
\end{equation*}
As in the previous cases, being $\det\mathbb{B}(i)$ is independent of $i$,
the drift velocity is the average of the speeds.
Then, setting $\mathbf{v}_{{}_{\textrm{drift}}}=0$, the diffusion matrix turns to be 
\begin{equation*}
		\mathbb{D}=\frac{2}{\mu_{12}}(\mathbf{v}^2\otimes\mathbf{v}^2)
				+\frac{2}{\mu_{13}}(\mathbf{v}^3\otimes\mathbf{v}^3)
				+\frac{2}{\mu_{14}}(\mathbf{v}^4\otimes\mathbf{v}^4)
				+\frac{2}{\mu_{15}}(\mathbf{v}^5\otimes\mathbf{v}^5).
\end{equation*}
Note, in this case, the absence of the dependency of $\mathbb{D}$ on the velocity $\mathbf{v}^1$.
\end{example}

From all the examples, in the case of symmetric transition matrix $\mathbb{B}$, two
specific features come evident: firstly, the drift speed is the average of the elements in
the speed set $\{\mathbf{v}^1,\dots\mathbf{v}^n\}$; secondly, the diffusion matrix is described
as a linear combination with non-negative coefficients of tensor of the form $w\otimes w$.
This latter property, in particular, implies that the matrix $\mathbb{D}$ is non-negative definite.

In what follows, we show that the same hallmark is shared by any system \eqref{discrkin}
if the matrix $\mathbb{B}$ is symmetric as a consequence of the special structure of its
principal minors.

\begin{proposition}\label{prop:symdrift}
If the transition matrix $\mathbb{B}$ is symmetric, then $\det\mathbb{B}(i)=\det \mathbb{B}(j)$ 
for any $i,j\in\{1,\dots,n\}$.
\end{proposition}

\begin{proof}
Substituting the first row with the sum of all the rows, using symmetry and, then, the first column with 
the sum of all the columns
\begin{equation*}
	\det \mathbb{B}(1)=\det\begin{pmatrix} \sum\limits_{j\neq 2} \mu_{j2} & -\mu_{23} & \dots \\
 					-\mu_{32} & \sum\limits_{j\neq 3} \mu_{j3} & \dots \\
 					\vdots & \vdots & \ddots \end{pmatrix}
		=\det\begin{pmatrix} \mu_{12} & \mu_{13} & \dots \\
 					-\mu_{32} & \sum\limits_{j\neq 3} \mu_{j3} & \dots \\
 					\vdots & \vdots & \ddots  \end{pmatrix} \\
\end{equation*}\begin{equation*}
		=\det\begin{pmatrix} \mu_{12} & \mu_{13} & \dots \\
 		 		-\mu_{23} & \sum\limits_{j\neq 3} \mu_{j3} & \dots \\
 				\vdots & \vdots & \ddots \end{pmatrix}
		=\det\begin{pmatrix} \sum\limits_{j\neq 1}\mu_{1j} & \mu_{13} & \dots \\
 				\mu_{13} & \sum\limits_{j\neq 3} \mu_{j3} & \dots \\
 				\vdots & \vdots & \ddots \end{pmatrix}
\end{equation*}\begin{equation*}
		=-\det\begin{pmatrix} -\sum\limits_{j\neq 1}\mu_{1j} & -\mu_{13} & \dots \\
 				\mu_{13} & \sum\limits_{j\neq 3} \mu_{j3} & \dots \\
 		 		\vdots & \vdots & \ddots \end{pmatrix}
		=\det\begin{pmatrix} \sum\limits_{j\neq 1}\mu_{1j} & -\mu_{13} & \dots \\
 		 		-\mu_{13} & \sum\limits_{j\neq 3} \mu_{j3} & \dots \\
 				\vdots & \vdots & \ddots \end{pmatrix} 
			= \det \mathbb{B}(2).
\end{equation*}
The proof is complete.
\end{proof}

As a consequence of the independency of the principal minors of order 1 of the matrix $\mathbb{B}$,
as suggested by the previous examples, the form of the drift velocity becomes particularly simple.

\begin{theorem}\label{thm:drift}
If the transition matrix $\mathbb{B}$ is symmetric, there holds
\begin{equation}\label{driftsym}
	\mathbf{v}_{{}_{\textrm{drift}}}=\frac{1}{n}\sum_{i=1}^n \mathbf{v}^i.
\end{equation}
\end{theorem}

\begin{proof} The proof is straightforward.
Thanks to Proposition \ref{prop:symdrift}, $\det\mathbb{B}(i)$ is independent on $i$
and thus $I_1(\mathbb{B})=n\det\mathbb{B}(i)$.
Then, from \eqref{drift} follows
\begin{equation*}
	(\mathbf{v}_{{}_{\textrm{drift}}})_{j}
	=\frac{\det\mathbb{B}(i)}{n\det\mathbb{B}(i)}
		\bigl(\mathbf{1}\cdot(v^1_j,\dots,v^n_j)\bigr)
	=\frac{1}{n}\sum_{i=1}^n (\mathbf{v}^i)_{j}
\end{equation*}
where $\mathbf{1}=(1,\dots,1)$.
\end{proof}

In order to get a deeper understanding on the structure on the diffusion matrix $\mathbb{D}$
relative to \eqref{discrkin}, it is needed a more precise understanding of the form of
the principal minors of the transition matrix $\mathbb{B}$.
In this direction, a fundamental tools is provided by a generalization, proved in \cite{Chai82},
of the {\bf Kirchoff's matrix tree Theorem}, a well-known result in graph theory.
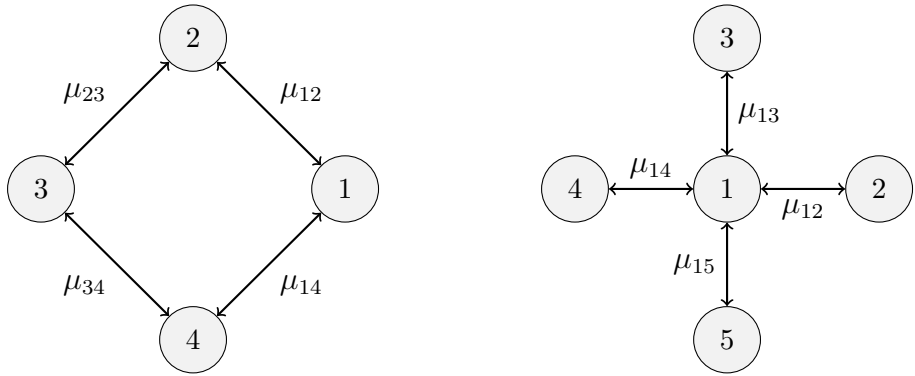
\begin{figure}[hbt] 
\begin{tikzpicture}[scale=2.0, auto, swap]
    \foreach \pos/\name in {{(1,0)/1}, {(0,1)/2}, {(-1,0)/3},{(0,-1)/4}}
        \node[vertex] (\name) at \pos {$\name$};
    \foreach \source/ \dest /\weight in {1/2/\mu_{12}, 2/3/\mu_{23},3/4/\mu_{34},4/1/\mu_{14}}
        \path[edge] (\source) -- node[weight] {$\weight$} (\dest);
 \end{tikzpicture}
\hskip2cm
\begin{tikzpicture}[scale=2.0, auto, swap]
    \foreach \pos/\name in {{(0,0)/1}, {(1,0)/2}, {(0,1)/3}, {(-1,0)/4},{(0,-1)/5}}
        \node[vertex] (\name) at \pos {$\name$};
    \foreach \source/ \dest /\weight in {1/2/\mu_{12}, 1/3/\mu_{13},1/4/\mu_{14},1/5/\mu_{15}}
        \path[edge] (\source) -- node[weight] {$\weight$} (\dest);
 \end{tikzpicture}
\caption{\footnotesize The graph relative to the Examples \ref{ex:fourside}--\ref{ex:fourtree}.}
\end{figure}

In its original version, this result affirms that the number of spanning trees of a given 
graph coincides with the determinant of an appropriate matrix associated to the graph.
Here, we use this fascinating connection the other way round: given the transition
matrix $\mathbb{B}$, we consider a corresponding graph and determine the values
of its minors by means of trees contained in the graph.
Indeed, the velocity changes dictated by the transition matrix $\mathbb{B}$ can be equivalently
represented by means of a directed graph, usually shortly called a {\bf digraph}, whose vertices
are the speeds $\mathbf{v}^i$, and with arcs from the $i$-th to the $j$-th node weighted by the 
transition rate $\mu_{ij}$.
The symmetry assumption on $\mathbb{B}$ translates into the fact that the graph is actually
undirected with weights given by the common values of the transition rates.
In this respect, Example \ref{ex:three} correspond to the graph depicted in Figure 1 
and Examples \ref{ex:fourside}--\ref{ex:fourtree} to the graphs in Figure 2.
Also, irreducibility of the matrix $\mathbb{B}$ is equivalent to the connectedness of the associated graph.

The collection of velocities $\mathbf{v}^i$ with arcs weighted by the rate $\mu_{ij}$
will be referred to as the {\bf graph associated to \eqref{discrkin}}.
Let us stress that the graph representation of the vertex is not directly related
with the effective value of the velocity $\mathbf{v}^i$ as vector in $\R^d$.
The content of the graph is only illustrative on the admissibile speed transitions.

To proceed, let us also recall that a {\bf tree} is an undirected graph in which any two
vertices are connected by exactly one simple path.
Equivalently, a tree is a connected graph without simple cycles.
A {\bf forest} is a disjoint union of trees.
Finally, given a forest $F$ in the graph $\Gamma$, we denote by $\mu_{F}$ the product
of the weights of all the arcs in $F$ and we call it {\bf weight of the forest}.
By definition, if $F$ is composed by a single vertex, its weight is equal to 1.

We are ready to state a (facilitated) version of the {\it (All minors) matrix tree theorem}
proved in \cite{Chai82}.

\begin{theorem}\label{thm:allminors}
Let $\Gamma=(\mathbf{v}^i,\mu_{ij})$ be the graph associated to \eqref{discrkin}.
Let the transition matrix $\mathbb{B}$ be symmetric and let $\mathcal{I}=\{i_1<\dots<i_k\}$ be a set
of indeces in $\{1,\dots,n\}$. Then
\begin{equation*}
	\det \mathbb{B}(\mathcal{I})=\sum_{F\in \mathcal{F}} \mu_{F}
\end{equation*}
where $\mathcal{F}$ is the set of forests $F$ in $\Gamma$ such that\par
{\bf i.} all the verteces in $\Gamma$ are contained in $F$;\par
{\bf ii.} $F$ contains exactly $k$ trees;\par
{\bf iii.} each tree in $F$ contains exactly one vertex $\mathbf{v}^i$ with $i\in\mathcal{I}$.
\end{theorem}
\begin{figure}[hbt] 
\begin{tikzpicture}[scale=2.75, auto, swap]
    \foreach \pos/\name in {{(1,0)/1}, {(0.71,0.71)/2}, {(0,1)/3},{(-0.71,0.71)/4},
    					{(-1,0)/5}, {(-0.71,-0.71)/6}, {(0,-1)/7},{(0.71,-0.71)/8}}
        \node[vertex] (\name) at \pos {$\name$};
    \foreach \source/ \dest /\weight in {1/2/\mu_{12}, 2/3/\mu_{23},3/4/\mu_{34},4/5/\mu_{45},
    							5/6/\mu_{56}, 6/7/\mu_{67},7/8/\mu_{78},8/1/\mu_{18}}
        \path[edge] (\source) -- node[weight] {$\weight$} (\dest);
 \end{tikzpicture}
 \caption{\footnotesize The graph relative to the Example \ref{ex:tour} in the case $n=8$.}
\end{figure}

Formulas \eqref{drift} and \eqref{diffmatrix} can be now re-interpreted taking advantage
of Theorem \ref{thm:allminors} in the cases $k=1$ and $k=2$.
Indeed, to compute the value $\det\mathbb{B}(i)$ (which is independent on $i$ thanks
to Proposition \ref{prop:symdrift}) it is sufficient to consider all the trees composed by 
all the verteces $\mathbf{v}^i$ and contained in the graph associated to \eqref{discrkin}
and to compute the sum of the weights of such trees.

\begin{example}\label{ex:tour} \normalfont
As an explicative example, let us consider the case of $n$ velocities $\mathbf{v}^i$ 
with rate $\mu_{ij}$ symmetric and positive if and only if, for $i<j$, either $j=i+1$
or $i=1$ and $j=n$ (see Figure 3).
The graph associated to such choice is a cycle and all of its trees containing all of the $n$
vertices are obtained by removing a single arc. Thus
\begin{equation*}
	\det\mathbb{B}(i)=\sum_{j=1}^{n}\frac{\mu_{12}\mu_{23}\dots\mu_{n-1,n}\mu_{n1}}{\mu_{j,j+1}}
\end{equation*}
where $\mu_{n,n+1}=\mu_{n1}$.
\end{example}

\begin{example}\label{ex:tree} \normalfont
Somewhat oppositely with respect to Example \ref{ex:tour}, we can consider a matrix
$\mathbb{B}$ such that transitions are possibile only from and towards a given specific
speed, say $\mathbf{v}^1$, i.e. $\mu_{ij}>0$ for $i<j$ if and only if $i=1$ (see Figure 4).
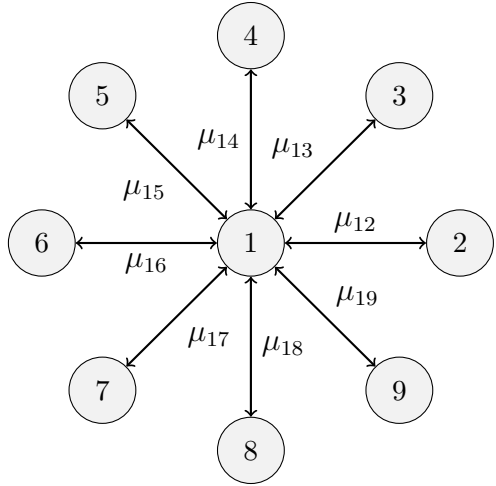
\begin{figure}[hbt] 
\begin{tikzpicture}[scale=2.75, auto]
    \foreach \pos/\name in {{(0,0)/1}, {(1,0)/2}, {(0.71,0.71)/3}, {(0,1)/4},{(-0.71,0.71)/5},
    					{(-1,0)/6}, {(-0.71,-0.71)/7}, {(0,-1)/8},{(0.71,-0.71)/9}}
        \node[vertex] (\name) at \pos {$\name$};
    \foreach \source/ \dest /\weight in {1/2/\mu_{12}, 1/3/\mu_{13},1/4/\mu_{14},1/5/\mu_{15},
    							1/6/\mu_{16}, 1/7/\mu_{17},1/8/\mu_{18},1/9/\mu_{19}}
        \path[edge] (\source) -- node[weight] {$\weight$} (\dest);
 \end{tikzpicture}
\caption{\footnotesize The graph relative to the Example \ref{ex:tour} in the case $n=9$.}
\end{figure}
The graph is in this case a tree. 
Hence, $\det\mathbb{B}(i)$ is the product of all the rates $\mu_{1j}$ with $j\in\{2,\dots,n\}$.
\end{example}

Next, we take advantage of Theorem \ref{thm:allminors} to give a different representation
of the diffusion matrix $\mathbb{D}$ defined in \eqref{diffmatrix}.
Given the graph $\Gamma$ associated to \eqref{discrkin}, we set
\begin{equation*}
	\mathcal{F}_2:=\{\textrm{forests $F$ partitioning the vertex of $\Gamma$ into two trees}\}.
\end{equation*}
Any element $F$ in $\mathcal{F}_2$ is composed by two trees, $T^1$ and $T^2$,
ordered according to the size of the smallest node in each tree.
Additionaly, given indeces $i,j$, with $i<j$, we use the notation
\begin{equation*}
	\mathcal{F}_2(i,j):=\{F=\{T^1,T^2\}\in\mathcal{F}_2\,:\,
		\textrm{$i$ and $j$ are verteces of $T^1$ and $T^2$, respectively}\}.
\end{equation*}
Then, given indeces $i$ and $j$, there holds
\begin{equation*}
	\det\mathbb{B}(i,j)=\sum_{\mathcal{F}_2(i,j)}\mu_{{}_{T^1}}\mu_{{}_{T^2}}.
\end{equation*}
and thus, for $\mathbf{v}_{{}_{\textrm{drift}}}=0$,
\begin{equation*}
	\mathbb{D}=-\frac{2}{I_1(\mathbb{B})}
		\sum_{h<k} \sum_{\mathcal{F}_2(h,k)}\mu_{{}_{T^1}}\mu_{{}_{T^2}}
		\,(\mathbf{v}^h\otimes \mathbf{v}^k)^*.
\end{equation*}
Reversing the order of the sums, we infer
\begin{equation*}
	\begin{aligned}
	\mathbb{D}&=-\frac{2}{I_1(\mathbb{B})}\sum_{\mathcal{F}_2}  \mu_{{}_{T^1}}\mu_{{}_{T^2}}
		\sum_{h\in T^1} \sum_{k\in T^2}  (\mathbf{v}^h\otimes \mathbf{v}^k)^*\\
			&=-\frac{2}{I_1(\mathbb{B})}\sum_{\mathcal{F}_2}  \mu_{{}_{T^1}}\mu_{{}_{T^2}}
		   \left(\sum_{h\in T^1}\mathbf{v}^h\otimes \sum_{k\in T^2}\mathbf{v}^k\right)^*
	\end{aligned}
\end{equation*}
Thanks to Proposition \ref{prop:symdrift}, $\mathbf{v}_{{}_{\textrm{drift}}}=0$ if and only if
the sum of the velocities $\mathbf{v}^i$ is null.
As a consequence, we obtain the proof of the following statement, the more intriguing
contribution of the present paper, giving an explicit formula for the diffusion matrix in
the symmetric setting.

\begin{theorem}\label{thm:finaldiff}
Let the transition matrix $\mathbb{B}$ be symmetric. Then, denoting by $\mathbf{w}(T)$
the sum of the velocities $\mathbf{v}^i$ in a given tree $T$ of the graph associated to \eqref{discrkin},
the diffusion matrix relative to \eqref{discrkin} is 
\begin{equation}\label{finaldiff}
	\mathbb{D}=\frac{2}{I_1(\mathbb{B})}\sum_{\mathcal{F}_2}  \mu_{{}_{T^1}}\mu_{{}_{T^2}}
		   \left(\mathbf{w}(T_1)\otimes \mathbf{w}(T_1)\right).
\end{equation}
In particular, the matrix $\mathbb{D}$ is non-negative definite.
\end{theorem}

To get acquainted with formula \eqref{finaldiff}, let us consider some of the previous Examples.

\begin{example}\label{ex:fourside2} \normalfont
For what concerns the system of Example \ref{ex:fourside}, the set $\mathcal{F}_2$ is composed
by six elements, obtained by removing two arcs of the corresponding graph (see Figure 5).
\begin{figure}[hbt] 
\begin{tikzpicture}[scale=1.6, auto, swap]
    \foreach \pos/\name in {{(1,0)/1}, {(0,1)/2}, {(-1,0)/3},{(0,-1)/4}}
        \node[vertex] (\name) at \pos {$\name$};
    \foreach \source/ \dest /\weight in {2/3/\mu_{23},3/4/\mu_{34}}
        \path[edge] (\source) -- node[weight] {$\weight$} (\dest);
 \end{tikzpicture}
\hskip1cm
\begin{tikzpicture}[scale=1.6, auto, swap]
    \foreach \pos/\name in {{(1,0)/1}, {(0,1)/2}, {(-1,0)/3},{(0,-1)/4}}
        \node[vertex] (\name) at \pos {$\name$};
    \foreach \source/ \dest /\weight in {1/4/\mu_{14}, 3/4/\mu_{34}}
        \path[edge] (\source) -- node[weight] {$\weight$} (\dest);
 \end{tikzpicture}
\hskip1cm
\begin{tikzpicture}[scale=1.6, auto, swap]
    \foreach \pos/\name in {{(1,0)/1}, {(0,1)/2}, {(-1,0)/3},{(0,-1)/4}}
        \node[vertex] (\name) at \pos {$\name$};
    \foreach \source/ \dest /\weight in {1/2/\mu_{12},4/1/\mu_{14}}
        \path[edge] (\source) -- node[weight] {$\weight$} (\dest);
 \end{tikzpicture}
\vskip.5cm
\begin{tikzpicture}[scale=1.6, auto, swap]
    \foreach \pos/\name in {{(1,0)/1}, {(0,1)/2}, {(-1,0)/3},{(0,-1)/4}}
        \node[vertex] (\name) at \pos {$\name$};
    \foreach \source/ \dest /\weight in {1/2/\mu_{12}, 2/3/\mu_{23}}
        \path[edge] (\source) -- node[weight] {$\weight$} (\dest);
 \end{tikzpicture}
\hskip1cm
\begin{tikzpicture}[scale=1.6, auto, swap]
    \foreach \pos/\name in {{(1,0)/1}, {(0,1)/2}, {(-1,0)/3},{(0,-1)/4}}
        \node[vertex] (\name) at \pos {$\name$};
    \foreach \source/ \dest /\weight in {1/2/\mu_{12}, 3/4/\mu_{34}}
        \path[edge] (\source) -- node[weight] {$\weight$} (\dest);
 \end{tikzpicture}
\hskip1cm
\begin{tikzpicture}[scale=1.6, auto, swap]
    \foreach \pos/\name in {{(1,0)/1}, {(0,1)/2}, {(-1,0)/3},{(0,-1)/4}}
        \node[vertex] (\name) at \pos {$\name$};
    \foreach \source/ \dest /\weight in {2/3/\mu_{23},4/1/\mu_{14}}
        \path[edge] (\source) -- node[weight] {$\weight$} (\dest);
 \end{tikzpicture}
\caption{\footnotesize The six elements of  $\mathcal{F}_2$ for the graph of the Example \ref{ex:fourside}.}
\end{figure}
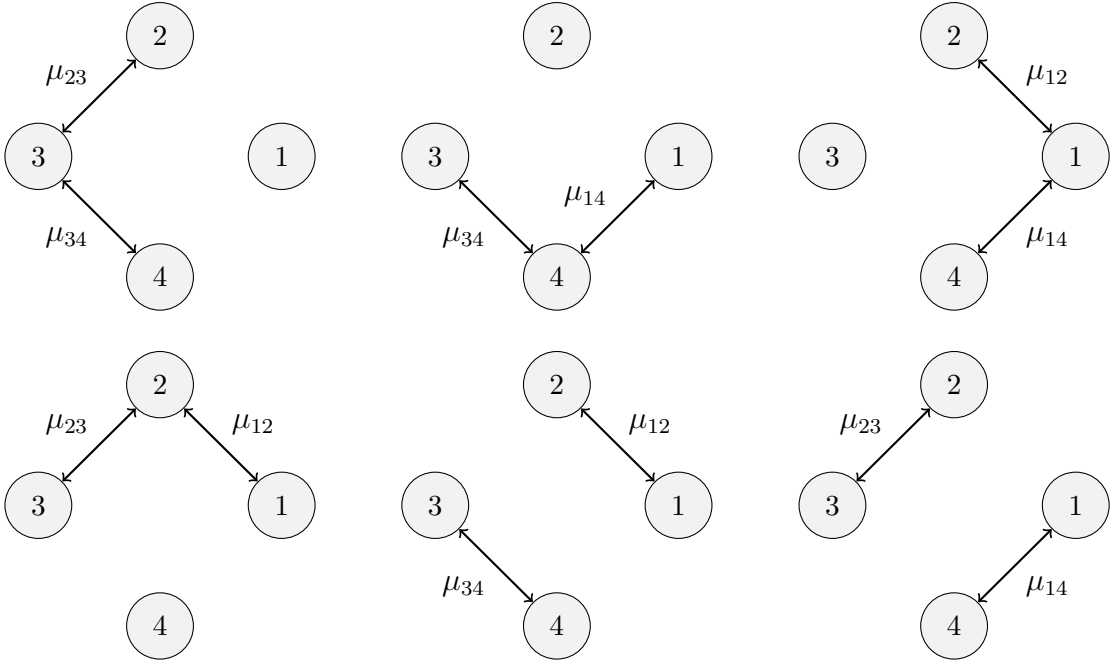
Each of the element $F=\{T^1,T^2\}$ in $\mathcal{F}_2$ gives the coefficient of an appropriate matrix of the form
$\mathbf{w}\otimes\mathbf{w}$  in the formula for the diffusion matrix $\mathbb{B}$.
The vector $\mathbf{w}$ is obtained by summing the verteces of the tree $T^1$, or of the tree $T^2$.
These two sum of verteces coincides because of the requirement $\mathbf{v}_{{}_{\textrm{drift}}}=0$.
The final expression of $\mathbb{D}$ has been given in Example \ref{ex:fourside}.
\end{example}

\begin{example}\label{ex:tour2} \normalfont
Example \ref{ex:tour} generalizes the previous one.
The family $\mathcal{F}_2$ is completely described by the trees $T^1$ containing a given vertex,
say $\mathbf{v}^1$.
Since such threes are composed by path of $k$ consecutive arcs, for $k=1,\dots,n-1$,
there are exactly $k$ of such threes containing the vertex.
Therefore, the family $\mathcal{F}_2$ is composed by $\frac{n(n-1)}{2}$ elements.
The diffusion matrix $\mathbb{D}$ is the sum of weighted matrix $\mathbf{w}\otimes\mathbf{w}$
where $\mathbf{w}$ is the sum of the verteces in the tree $T^1$ and the weights is the product
of the arcs in $T^1$ and in $T^2$ (thus the product of all the coefficients $\mu_{ij}$ with the exception
of the one corresponding to the two arcs not in $T^1$ and $T^2$).
\end{example}

\begin{example}\label{ex:tree2} \normalfont
The situation for Example \ref{ex:tree} is particularly simple.
Indeed, the set $\mathcal{F}_2$ is composed by $n-1$ element each of which correspond
to the forest of a tree composed by a single vertex $\mathbf{v}^{i}$, $i\in\{2,\dots,n\}$
and the tree, denoted by $T(i)$ composed by the remaining part of the graph excluding
both the vertex $\mathbf{v}^{i}$ and the arc from $\mathbf{v}^{1}$ to $\mathbf{v}^{i}$.
Thus, the diffusion matrix takes the form
\begin{equation*}
	\mathbb{D}=\frac{2}{\mu_{12}\dots\mu_{1n}}\sum_{i=2}^{n}
			\mu_{{}_{T(i)}}\left(\mathbf{v}^i\otimes \mathbf{v}^i\right)
			=\sum_{i=2}^{n}\frac{2}{\mu_{1i}}\left(\mathbf{v}^i\otimes \mathbf{v}^i\right).
\end{equation*}
Apparently,  the formula for $\mathbb{D}$ does not depend on the choice of $\mathbf{v}^1$,
but it should be always kept in mind that formula \eqref{finaldiff} gives the diffusion matrix 
when the drift term $\mathbf{v}_{{}_{\textrm{drift}}}$ has been set to zero by applying a
change in the reference frame.
Thus, the speed $\mathbf{v}^1$ cannot be chosen independently on the other speeds.
\end{example}

To conclude the Section let us consider a specific case for which the expression of the
diffusion matrix simplifies further, involving only term of the form $\mathbf{v}^i\otimes\mathbf{v}^i$.

\begin{proposition}\label{prop:onlyvi}
Let $\mathbb{B}$ be a symmetric matrix and assume:\par
	{\bf i.} either $n=2k$ for some $k$ and $\mathbf{v}^{2\ell}=-\mathbf{v}^{2\ell-1}$
		for any $\ell\in\{1,\dots,k\}$;\par
	{\bf ii.} or $n=2k+1$ for some $k$, $\mathbf{v}^{2\ell}=-\mathbf{v}^{2\ell-1}$ for any $\ell\in\{1,\dots,k\}$
		and $\mathbf{v}^{2k+1}=0$.\\
If the transition matrix $\mathbb{B}$ is such that
\begin{equation}\label{ip2minor}
	\mathbb{B}(2i-1,j)=\mathbb{B}(2i,j)\qquad\textrm{for any }i=1,\dots,k,\quad j=1,\dots,n,
\end{equation}
then there holds
\begin{equation}\label{diffmatrix2}
	\mathbb{D}=-\frac{2}{I_1(\mathbb{B})}\sum_{i=1}^{m}
		\sum_{\mathcal{F}_2(2i-1,2i)}\mu_{{}_{T^1}}\mu_{{}_{T^2}}(\mathbf{v}^{2i}\otimes \mathbf{v}^{2i})
\end{equation}
\end{proposition}

\begin{proof}
Let us deal with case {\bf i.}, the other one being similar.
The proof consists in showing that formula \eqref{diffmatrix}, under assumption \eqref{ip2minor},
can be rearranged as
\begin{equation*}
	\mathbb{D}=-\frac{2}{I_1(\mathbb{B})}
		\sum_{i=1}^{k} \mathbb{B}(2i-1,2i)\,(\mathbf{v}^{2i}\otimes \mathbf{v}^{2i}).
\end{equation*}
Let us proceed by induction.
For $k=1$, there holds
\begin{equation*}
	\frac{1}{2}\,I_1(\mathbb{B})\,\mathbb{D}=-\mathbb{B}(1,2)\,(\mathbf{v}^1\otimes \mathbf{v}^2)^*
		=\mathbb{B}(1,2)(\mathbf{v}^2\otimes \mathbf{v}^2)
\end{equation*}
Then, assuming that the thesis holds for $k-1$, we infer
\begin{equation*}\begin{aligned}
	&\frac12\,I_1(\mathbb{B})\,\mathbb{D}=-\sum_{j>1} \mathbb{B}(1,j)\,(\mathbf{v}^1\otimes \mathbf{v}^j)^*
		-\sum_{j>2} \mathbb{B}(2,j)\,(\mathbf{v}^2\otimes \mathbf{v}^j)^*
		-\sum_{2<i<j} \mathbb{B}(i,j)\,(\mathbf{v}^i\otimes \mathbf{v}^j)^*\\
	&\quad=-\mathbb{B}(1,2)\,(\mathbf{v}^1\otimes \mathbf{v}^2)^*
		+\sum_{j>2} \bigl(\mathbb{B}(1,j)-\mathbb{B}(2,j)\bigr)(\mathbf{v}^2\otimes \mathbf{v}^j)^*
		-\sum_{2<i<j} \mathbb{B}(i,j)\,(\mathbf{v}^i\otimes \mathbf{v}^j)^*\\
	&\quad=\mathbb{B}(1,2)\,(\mathbf{v}^2\otimes \mathbf{v}^2)^*
		+\sum_{i=2}^{k} \mathbb{B}(2i-1,2i)\,(\mathbf{v}^{2i}\otimes \mathbf{v}^{2i})
		=\sum_{i=1}^{k} \mathbb{B}(2i-1,2i)\,(\mathbf{v}^{2i}\otimes \mathbf{v}^{2i}),
\end{aligned}\end{equation*}
that gives the conclusion.
\end{proof}

In term of transitions, hypothesis \eqref{ip2minor} asserts that the probability to jump on/from 
velocity $\mathbf{v}^{2j}$ or $-\mathbf{v}^{2j}$ to a different given speed $\mathbf{v}^{\ell}$ is the same.
Such assumption is meaningful either when each couple of speeds $\mathbf{v}^{2j-1},\mathbf{v}^{2j}$
corresponds to a cartesian direction or also in the case of ``undirected'' motion in the sense that
the two directions on the same line have same probability of success after transition.
As an example, the latter situation is of interest in the modeling of undirected tissues as considered in
\cite{Hill06} in the modeling of mesenchymal motion.

\section{Asymptotically parabolic behavior}\label{sect:asym}

To corroborate the analysis of the previous Section, we now show that the representation of the diffusion
matrix $\mathbb{D}$ is indeed significant for the description of the large-time behavior of solutions
to \eqref{discrkin}--\eqref{initialdatum}.

As mentioned in the Introduction, the main result of this Section fits into a well-estabilished research
strand, whose main target is to quantify the large-time parabolic behavior of solutions of a class of
dissipative hyperbolic equations. 
Specifically, the main point of the analysis is to show that the distance of the hyperbolic solution 
and some solution to a corresponding parabolic problem, decayes to zero
faster than the decay of each separate term, showing that the dissipation mechanism is 
asymptotically of the same type. 
Being, usually, the parabolic behavior preferred for regularity reasons, it is usually stated
that the hyperbolic equation has an {\bf asymptotically parabolic nature}. 

To quote some contributions, without any intention of completeness, let us mention
the direction that explores the form of the Green function with ref.\cite{Zeng99},
motivated by nonequilibrium gasdynamics, and ref. \cite{BianHanoNata07}, where a general 
class of relaxation systems is explored in details;
the $L^\infty$-bound in ref. \cite{YangMila00} and the $L^p-L^q$ estimates
in reff. \cite{MarcNish03, Nish03} (and descendants) relative to the prototypical
case of the relation between heat and telegraph equation;
the analysis $L^2$-asymptotic expansions of the solutions for the heat and the damped wave
equation proposed in ref. \cite{Volk10}, which clearly shows how the diffusive behavior is the effect
of the cancellation of leading order terms.

Concerning the diffusive behavior of system \eqref{discrkin} as $t\to+\infty$,
we state and prove the following result.

\begin{theorem}\label{thm:asymptotic}
Assume the transition matrix $\mathbb{B}$ to be symmetric with strictly positive diagonal elements
and the diffusion matrix $\mathbb{D}$ to be positive definite.
Let $\mathbf{f}$ be the solution to \eqref{discrkin}--\eqref{initialdatum}, $u$ be the sum
of the components $f_i$, i.e. $u:=\mathbf{1}\cdot \mathbf{f}$, and $u_{{}_{\textrm{par}}}$
be the solution to
\begin{equation*}
	\frac{\partial w}{\partial t}
		=\textrm{\normalfont div}\left(\mathbb{D}\nabla_{\mathbf{x}} w\right)
	\qquad
	w(x,0)=\mathbf{1}\cdot \mathbf{f}_0.
\end{equation*} 
Then, if $\mathbf{f}_0\in [L^1\cap L^2(\R^d)]^n$, there holds
\begin{equation}\label{asymptotic}
	|u-u_{{}_{\textrm{par}}}|_{{}_{L^2}}(t)
		\leq C\,t^{-\frac{1}{4}d-\frac{1}{2}}|{\mathbf{f}}_0|_{{}_{L^1\cap L^2}}
\end{equation}
for some $C>0$ (independent of $t$ and ${\mathbf{f}}_0$).
\end{theorem}

Hypotheses on the matrices $\mathbb{B}$ and $\mathbb{D}$ are satisfied when
the conditions required in Proposition \ref{prop:sk} hold, namely
\begin{equation*}
	\mathbb{B}\quad\textrm{irreducible}
	\qquad\textrm{and}\qquad
	\linspan\{\mathbf{v}^i-\mathbf{v}^j\,:\, i,j=1,\dots,n\}=\R^d.
\end{equation*}
Indeed, if the matrix $\mathbb{B}$ is irreducible (or, equivalent, if the graph associated
to the system \eqref{discrkin} is connected), the diagonal elements are strictly positive.
Moreover, since in \cite{ShizKawa85} it is proved that the property in Proposition \ref{prop:sk}
is equivalent to the bound
\begin{equation*}
	\Real\lambda\leq -\frac{c_0|\mathbf{k}|^2}{1+|\mathbf{k}|^2}
\end{equation*}
for some $c_0>0$ and for any $(\lambda,\mathbf{k})$ satisfying the dispersion relation,
the matrix $\mathbb{D}$ is forced to be positive definite, whenever the differences
$\mathbf{v}^i-\mathbf{v}^j$ generates all of $\R^d$.

To prove Theorem \ref{thm:asymptotic}, we applying Fourier transform to \eqref{discrkin}.
Denoting by $\hat{u}$ the Fourier transform of function $u$, we obtain a system of ordinary
differential equations for the frequency variables $\hat f=\hat f(\mathbf{k},t)$
\begin{equation}\label{discrkinF}
  	\frac{\partial \hat f_i}{\partial t}+(\mathbf{v}^i\cdot\mathbf{k})\hat f_i
  		+\sum_{j\neq i}\bigl(\mu_{ij}\,\hat f_i-\mu_{ji}\,\hat f_j\bigr)=0, \qquad i=1,\dots,n,
\end{equation}
with initial conditions
\begin{equation}\label{initialF}
 \hat f_i(\mathbf{k},0)=\hat f_{0,i}(\mathbf{k})\qquad\qquad \mathbf{k}\in\mathbb{R}^d,\quad i=1,\dots,n
\end{equation}
Denoting by $\hat{\mathbf{f}}$ and $\hat{\mathbf{f}}_0$ the vectors of components $\hat f_i$ and $\hat f_{0,i}$,
respectively, the solution to \eqref{discrkinF}--\eqref{initialF} is given by
\begin{equation*}
	 \hat{\mathbf{f}}(\mathbf{k},t)=\exp\left\{-\bigl(\diag(\mathbf{v}^i\cdot\mathbf{k})
	 	+\mathbb{B}\bigr)t\right\} \hat{\mathbf{f}}_0(\mathbf{k})
\end{equation*}
where the symbol $\mathbb{B}$ denotes the same matrix of Section \ref{sect:drift}.
Let us consider the functions
\begin{equation*}
	 \hat{u}(\mathbf{k},t)=\mathbf{1}\cdot \exp\left\{-\bigl(\diag(\mathbf{v}^i\cdot\mathbf{k})
	 	+\mathbb{B}\bigr)t\right\} \hat{\mathbf{f}}_0(\mathbf{k})
	\quad\textrm{and}\quad
	\hat{u}_{{}_{\textrm{par}}}(\mathbf{k},t)
	 	=e^{(\mathbf{k}\cdot\mathbb{D}\mathbf{k})t}\hat u_0(\mathbf{k})
\end{equation*}
where $\hat u_0=\mathbf{1}\cdot\hat{\mathbf{f}}_0$.
We are interested in estimating the $L^2$-norm of the difference $\hat{u}-\hat{u}_{{}_{\textrm{par}}}$.

Let $\varepsilon_0>0$ to be chosen. Then, there holds
\begin{equation}\label{diffest}
	|\hat{u}-\hat{u}_{{}_{\textrm{par}}}|_{{}_{L^2}}^2(t)\leq I_1(t)+I_2(t)
\end{equation}
where
\begin{equation*}
	\begin{aligned}
	I_1(t)&:=\int_{|\mathbf{k}|<\varepsilon_0} \Bigl|\mathbf{1}\cdot \exp\left\{-\bigl(\diag(\mathbf{v}^i\cdot\mathbf{k})
	 	+\mathbb{B}\bigr)t\right\} \hat{\mathbf{f}}_0(\mathbf{k})
		-e^{(\mathbf{k}\cdot\mathbb{D}\mathbf{k})t}(\mathbf{1}\cdot\hat{\mathbf{f}}_0)(\mathbf{k})\Bigr|^2\,d\mathbf{k}\\
	I_2(t)&:=\int_{|\mathbf{k}|\geq \varepsilon_0}\left\{|\hat{u}|^2
			+|\hat{u}_{{}_{\textrm{par}}}|^2\right\}(\mathbf{k},t)\,d\mathbf{k}
	\end{aligned}
\end{equation*}
The dispersion relation of \eqref{discrkin} in the limit $\mathbf k\to 0$ has been explored in Section \ref{sect:drift}.
In particular, in this regime, it is possible to identify the branch of solution of $p(\lambda,\mathbf{k})=0$
passing through the origin $(0,\mathbf{0})$ and to describe its  behavior by means of a scalar-valued
function $\lambda=\lambda(\mathbf{k})$ for which the following second-order expansion holds
\begin{equation*}
	\lambda(\mathbf{k})=\mathbf{k}\cdot\mathbb{D}\mathbf{k}+o(|\mathbf{k}|^2)
		\qquad\textrm{as}\quad \mathbf{k}\to 0.
\end{equation*}
Taking advantage of this relation, we are able to state a result concerning the bound of $I_1$.

\begin{lemma}\label{lemma:I1}
Let the vectors $\mathbf{v}^i$ be such that $\sum_{i} \mathbf{v}^i=0$
and let the matrix $\mathbb{D}$ be positive definite.
If $\mathbf{f}_0\in L^1\cap L^2(\R^d)$, then there exist $\varepsilon_0, C, c>0$
such that
\begin{equation}\label{estimateI1}
	I_1(t)\leq C\bigl(t^{-\frac{1}{2}d-1}|{\mathbf{f}}_0|^2_{{}_{L^1}}
		+e^{-ct}|{\mathbf{f}}_0|_{{}_{L^2}}^2\bigr).
\end{equation}
for any $t>0$.
\end{lemma}

\begin{proof}
Denoting by $\mathbb{P}(\mathbf{k})$ the spectral projector relative to $\lambda=\lambda(\mathbf{k})$,
we deduce that for $\varepsilon_0$ sufficiently small the matrix there holds
\begin{equation*}
	\exp\left\{-\bigl(\diag(\mathbf{v}^i\cdot\mathbf{k})+\mathbb{B}\bigr)t\right\}
		=e^{\lambda(\mathbf{k})t}\mathbb{P}(\mathbf{k})+O(e^{-\theta t})
\end{equation*}
for some $\theta>0$ uniform with respect to $\mathbf{k}$ such that $|\mathbf{k}|\leq \varepsilon_0$.
Then, the term $I_1$ can be estimated by
\begin{equation*}
		I_1(t)\leq \int_{|\mathbf{k}|<\varepsilon_0}
			\Bigl|\Bigl(e^{\lambda(\mathbf{k})t}\mathbf{1}\,\mathbb{P}(\mathbf{k})
				-e^{(\mathbf{k}\cdot\mathbb{D}\mathbf{k})t}\mathbf{1}\Bigr)
				\cdot\hat{\mathbf{f}}_0(\mathbf{k})\Bigr|^2\,d\mathbf{k}
			+O(e^{-2\theta t}) \int_{|\mathbf{k}|<\varepsilon_0}
				\bigl|\hat{\mathbf{f}}_0(\mathbf{k})\bigr|^2\,d\mathbf{k}.
\end{equation*}
The projection $\mathbb{P}(\mathbf{k})$ has the form $r(\mathbf{k})\otimes \ell(\mathbf{k})$ where $\ell$ and $r$
are, respectively, left and right eigenvectors of $\diag(\mathbf{v}^i\cdot\mathbf{k})+\mathbb{B}$ relative
to the eigenvalue $\lambda(\mathbf{k})$, normalized so that the condition $\ell(\mathbf{k})\cdot r(\mathbf{k})=1$
holds.
Since the sum of columns and rows of $\mathbb{B}$ is zero, by assumption, there holds
\begin{equation*}
	r(\mathbf{0})=\ell(\mathbf{0})=\frac{1}{\sqrt{n}}\mathbf{1}.
\end{equation*}
In particular, the zero-th order expansion for $\mathbb{P}(\mathbf{k})$ is
\begin{equation*}
	\mathbb{P}(\mathbf{k})=\frac{1}{n}(\mathbf{1}\otimes \mathbf{1})+O(\mathbf{k})
	\qquad\textrm{as}\quad \mathbf{k}\to \mathbf{0}.
\end{equation*}
Therefore, the term $I_1$ is bounded by
\begin{equation*}
		I_1(t)\leq I_{11}(t)+I_{12}(t)+ O(e^{-2\theta t})|\hat{\mathbf{f}}_0|_{{}_{L^2}}^2,
\end{equation*}
with
\begin{equation*}
	\begin{aligned}
	I_{11}(t)&:=o(1)\int_{|\mathbf{k}|<\varepsilon_0}e^{2(\mathbf{k}\cdot\mathbb{D}\mathbf{k})t}
			|\mathbf{k}|^4\bigl|\hat{u}_0(\mathbf{k})\bigr|^2\,d\mathbf{k},\\
	I_{12}(t)&:=O(1)\int_{|\mathbf{k}|<\varepsilon_0}
				e^{2\Real\lambda(\mathbf{k})t} |\mathbf{k}|^2|\hat{\mathbf{f}}_0(\mathbf{k})|^2\,d\mathbf{k}
	\end{aligned}
\end{equation*}
Since, for $v\in L^1(\R^d)$, there holds
\begin{equation*}
	\begin{aligned}
	\int_{|\mathbf{k}|<\varepsilon_0} e^{-2\theta|\mathbf{k}|^2 t}
		|\mathbf{k}|^{2\ell}|\hat{v}(\mathbf{k})|^2\,d\mathbf{k}
	&\leq |v|^2_{{}_{L^1}}\int_{|\mathbf{k}|<\varepsilon_0}
		e^{-2\theta|\mathbf{k}|^2 t}|\mathbf{k}|^{2\ell}\,d\mathbf{k}\\
	&=|v|^2_{{}_{L^1}}\,t^{-\frac{1}{2}d-\ell}\int_{|\mathbf{y}|<\varepsilon_0\sqrt{\theta t}}
		e^{-2|\mathbf{y}|^2}|\mathbf{y}|^{2\ell}\,d\mathbf{k}
	\leq C|v|^2_{{}_{L^1}}\,t^{-\frac{1}{2}d-\ell}.
	\end{aligned}
\end{equation*}
Collecting, we end up with 
\begin{equation*}
		I_1(t)\leq C\bigl(\,t^{-\frac{1}{2}d-2}|u_0|^2_{{}_{L^1}}
		+C\,t^{-\frac{1}{2}d-1}|{\mathbf{f}}_0|^2_{{}_{L^1}}
		+e^{-2\theta t}|{\mathbf{f}}_0|_{{}_{L^2}}^2\bigr),
\end{equation*}
which gives \eqref{estimateI1}.
\end{proof}

Concerning the term $I_2$, relative to value of $\mathbf{k}$ at positive distance from the origin $\mathbf{0}$,
the following estimate holds true.

\begin{lemma}\label{lemma:I2}
Let $\mathbf{v}^i$ be such that $\sum_{i} \mathbf{v}^i=0$ and let the matrix $\mathbb{B}$ be 
such that $\mu_{ii}>0$ for any $i=1,\dots,n$.
If $\mathbf{f}_0\in L^2(\R^d)$, then for any $\varepsilon_0>0$ there exist $C, c>0$ such that
\begin{equation}\label{estimateI2}
	I_2(t)\leq C\,e^{-ct}|{\mathbf{f}}_0|_{{}_{L^2}}^2,
\end{equation}
for any $t>0$.
\end{lemma}

\begin{proof}
Since $\mathbb{D}$ is positive definite,
from the definition of $\hat{u}_{{}_{\textrm{par}}}$ it follows
\begin{equation*}
	\int_{|\mathbf{k}|\geq \varepsilon_0} |\hat{u}_{{}_{\textrm{par}}}|^2 \,d\mathbf{k}
		\leq \int_{|\mathbf{k}|\geq \varepsilon_0} e^{-2c_0|\mathbf{k}|^2 t}|\hat u_0(\mathbf{k})|^2\,d\mathbf{k}
		\leq  e^{-2c_0\varepsilon_0^2 t} |\hat u_0|^2_{{}_{L^2}}
		=e^{-2c_0\varepsilon_0^2 t} |\mathbf{1}\cdot \mathbf{f}_0|^2_{{}_{L^2}}.
\end{equation*}
for some $c_0>$.
Thus, estimate \eqref{estimateI2} is proved if we show that there exists $\theta>0$ such that
\begin{equation}\label{stablelambda}
	\Real\lambda\leq -\theta<0
\end{equation}
for any $\lambda$ such that
$\det(\lambda\,\mathbb{I}+\diag(\mathbf{v}^i\cdot\mathbf{k})+\mathbb{B})=0$
for some $|\mathbf{k}|>\varepsilon_0$.

The description of the dispersion relation given in Section \ref{sect:drift} guarantees that
property $\Real\lambda<0$ holds for $\mathbf{k}\neq 0$ and small.

Next, we show that the system does not support pure imaginary value of $\lambda$
corresponding to purely imaginary values of $\mathbf{k}$.
Let $F$ be such that
\begin{equation}\label{evalue}
	 (\lambda+\mathbf{v}^i\cdot \mathbf{k})F_i+\sum_{j=1}^{n}\mu_{ij}\,F_j=0
 	\qquad\qquad i=1,\dots,n.
\end{equation}
By multiplying for the complex conjugate $\bar F_i$ and summing with respect to $i$, we get
\begin{equation*}
 	\sum_{i=1}^n (\lambda+v_i\cdot \mathbf{k})|F_i|^2+\sum_{i, j=1}^{n}\mu_{ij}\,\bar F_i\,F_j=0
\end{equation*}
Since the matrix $\mathbb{B}$ is symmetric, the last term is real;
hence, for $\lambda,\mathbf{k}$ purely imaginary,
\begin{equation*}
	 \sum_{i=1}^n (\lambda+\mathbf{v}^i\cdot \mathbf{k})|F_i|^2=\sum_{i, j=1}^{n}\mu_{ij}\,\bar F_i\,F_j=0
\end{equation*}
Since $F\in \ker B$, then $F=C\mathbf{1}$ for some $C\in\R$ and thus
\begin{equation*}
 	\lambda = -\frac{1}{n} \sum_{i=1}^n \mathbf{v}^i \cdot \mathbf{k}=0
\end{equation*}
At the moment, \eqref{stablelambda} holds for $|\mathbf{k}|\in[\varepsilon_0,M]$
for any $0<\varepsilon_0<M$ for some $\theta=\theta(\varepsilon_0, M)$.
The last step of the proof concerns with the high frequency regime $\mathbf{k}\in i\R^d$
with $|\mathbf{k}|\to\infty$.

To this aim, let us consider the eigenvalue problem \eqref{evalue} with
$\mathbf{k}=\varepsilon^{-1}\mathbf{h}$, $|\mathbf{h}|=1$ and $\lambda=\varepsilon^{-1}\nu$,
that is
\begin{equation}\label{evalue2}
	 (\mu+\mathbf{v}^i\cdot \mathbf{h})F_i+\varepsilon\sum_{j=1}^{n}\mu_{ij}\,F_j=0
 	\qquad\qquad i=1,\dots,n
\end{equation}
in the limit $\varepsilon\to 0$.
Setting
\begin{equation*}
	\nu=\nu_0+\varepsilon\,\nu_1+o(\varepsilon),\qquad
	F=F_0+\varepsilon\,F_1+o(\varepsilon)
\end{equation*}
plugging into \eqref{evalue2} and collecting the terms with same power of $\varepsilon$,
we get the relations
\begin{equation*}
	 (\nu_0+\mathbf{v}^i\cdot \mathbf{h})F_{0,i}=0
	 \qquad\textrm{and}\qquad
	  (\nu_0+\mathbf{v}^i\cdot \mathbf{h})F_{1,i}
	  +\nu_1 F_{0,i}+\sum_{j=1}^{n}\mu_{ij}\,F_{0,j}=0
\end{equation*}
Hence, we infer for the $0$-th order coefficients 
\begin{equation*}
	\nu_0=-\mathbf{v}^i\cdot \mathbf{h}
	\qquad\textrm{and}\qquad F_{0,i}=E_i
\end{equation*}
where $E_i$ denotes the $i$-th element of the canonical base of $\R^n$.
Plugging into the $1$-st order relation, we get the formula for the first coefficient
in the expansion of $\nu$,
\begin{equation*}
	\nu_1=-\mu_{ii}.
\end{equation*}
Coming back to the original variables $\lambda$ and $\mathbf{k}$, the asymptotic
expansions reads
\begin{equation*}
	\lambda(\mathbf{k})=-\mathbf{v}^i\cdot \mathbf{k}-\mu_{ii}+o(1)
	\qquad\textrm{as}\quad |\mathbf{k}|\to\infty.
\end{equation*}
Since the diagonal values $\mu_{ii}$ are assumed to be strictly positive, the bound
\eqref{stablelambda} can be prolunged to $m\to+\infty$, changing, if necessary, 
the value of the constant $\theta$.
\end{proof}

By means of the results of Lemmas \ref{lemma:I1}--\ref{lemma:I2}, the completion of
the proof of Theorem \ref{thm:asymptotic} is at hand.
Indeed, resuming from \eqref{diffest} and using \eqref{estimateI1}--\eqref{estimateI2},
we obtain
\begin{equation*}
	|\hat{u}-\hat{u}_{{}_{\textrm{par}}}|_{{}_{L^2}}^2(t)
		\leq C\bigl(t^{-\frac{1}{2}d-1}|{\mathbf{f}}_0|^2_{{}_{L^1}}
				+e^{-ct}|{\mathbf{f}}_0|_{{}_{L^2}}^2\bigr)
		\leq C\,t^{-\frac{1}{2}d-1}|{\mathbf{f}}_0|^2_{{}_{L^1\cap L^2}}
\end{equation*}
that gives \eqref{asymptotic} passing to the square roots and invoking Plancherel identity.

\end{document}